\documentclass[a4paper]{amsart}
\usepackage{amscd}
\usepackage{amsmath}
\usepackage{amssymb}
\usepackage{amsthm}
\usepackage{bbm} 
\usepackage{mathrsfs}
\usepackage{stmaryrd}

\usepackage[T1]{fontenc}

\usepackage{lipsum}

\usepackage{xcolor}

\makeatletter
\renewcommand\subsection{\@startsection{subsection}{2}%
  \z@{-.5\linespacing\@plus-.7\linespacing}{.5\linespacing}%
  {\normalfont\scshape}}
\renewcommand\subsubsection{\@startsection{subsubsection}{3}%
  \z@{.5\linespacing\@plus.7\linespacing}{-.5em}%
  {\normalfont\scshape}}
\makeatother

\makeatletter
\@namedef{subjclassname@2020}{%
  \textup{2020} Mathematics Subject Classification}
\makeatother

\numberwithin{equation}{section} \swapnumbers

\newtheorem{satz}{Satz}[section]

\newtheorem{theorem}[satz]{Theorem}
\newtheorem{proposition}[satz]{Proposition}
\newtheorem{corollary}[satz]{Corollary}
\newtheorem{lemma}[satz]{Lemma}

\newtheorem{remark}[satz]{Remark}

\newtheorem{example}[satz]{Example}

\newcommand{\bbr}{\mathbb{R}}
\newcommand{\bbe}{\mathbb{E}}
\newcommand{\bbn}{\mathbb{N}}
\newcommand{\bbp}{\mathbb{P}}
\newcommand{\bbq}{\mathbb{Q}}

\newcommand{\cald}{\mathcal{D}}

\newcommand{\calf}{\mathcal{F}}

\newcommand{\call}{\mathcal{L}}
\newcommand{\calm}{\mathcal{M}}

\newcommand{\lin}{{\rm lin}}

\newcommand{\diag}{{\rm diag}}

\newcommand{\bbI}{\mathbbm{1}}

\newcommand{\la}{\langle}
\newcommand{\ra}{\rangle}

\begin{document}

\title[Distance between closed sets and the solutions to SPDEs]{Distance between closed sets and the solutions to stochastic partial differential equations}
\author{Toshiyuki Nakayama \and Stefan Tappe}
\address{MUFG, Bank, Ltd., 2-7-1, Marunouchi, Chiyoda-ku, Tokyo 100-8388, Japan}
\email{tnkym376736@gmail.com}
\address{Albert Ludwig University of Freiburg, Department of Mathematical Stochastics, Ernst-Zermelo-Stra\ss{}e 1, D-79104 Freiburg, Germany}
\email{stefan.tappe@math.uni-freiburg.de}
\date{9 October, 2024}
\thanks{We are grateful to Marvin M\"{u}ller for bringing our attention to the AMS review of the paper \cite{Jachimiak-note}. Stefan Tappe gratefully acknowledges financial support from the Deutsche Forschungsgemeinschaft (DFG, German Research Foundation) -- project number 444121509.}
\begin{abstract}
The goal of this paper is to clarify when the solutions to stochastic partial differential equations stay close to a given subset of the state space for starting points which are close as well. This includes results for deterministic partial differential equations. As an example, we will consider the situation where the subset is a finite dimensional submanifold with boundary. We also discuss applications to mathematical finance, namely the modeling of the evolution of interest rate curves.
\end{abstract}

\maketitle\thispagestyle{empty}

\section{Introduction}

Consider a semilinear stochastic partial differential equation (SPDE) of the form
\begin{align}\label{SPDE-intro}
\left\{
\begin{array}{rcl}
dX(t) & = & \big( A X(t) + \alpha(X(t)) \big) dt + \sigma(X(t)) dW(t)\medskip
\\ X(0) & = & x
\end{array}
\right.
\end{align}
driven by a trace class Wiener process $W$. Let $K \subset H$ be a closed subset of the state space $H$, which we assume to be a separable Hilbert space. Furthermore, let $A$ be the infinitesimal generator of a $C_0$-semigroup $(S_t)_{t \geq 0}$ on $H$, and let $\alpha$ and $\sigma$ be Lipschitz continuous coefficients.

In several papers stochastic invariance of $K$ has been studied; that is, necessary and sufficient conditions have been derived such that for every starting point $x \in K$ the mild solution $X(\,\cdot\,;x)$ to the SPDE (\ref{SPDE-intro}) stays in the set $K$. This has been done in \cite{Jachimiak} and, based on the support theorem presented in \cite{Nakayama-Support}, in \cite{Nakayama} for the case where $K$ is an arbitrary closed subset. In the particular case where $K$ is a closed convex cone, we refer to the papers \cite{Milian, FTT-positivity, Tappe-cones, Tappe-cones-2}; the latter three articles consider the situation where the SPDE (\ref{SPDE-intro}) is additionally driven by a Poisson random measure. If $K$ is a finite dimensional submanifold of $H$, the invariance problem has been studied in \cite{Filipovic-inv} and \cite{Nakayama} (see also the review article \cite{Tappe-Review}), and in \cite{FTT-manifolds} even in the situation where the SPDE (\ref{SPDE-intro}) has an additional Poisson random measure as driving noise. Here a related problem is the existence of a finite dimensional realization (FDR), which means that for each starting point $x \in H$ a finite dimensional invariant manifold $K \subset H$ with $x \in K$ exists. This problem has mostly been studied for the so-called Heath-Jarrow-Morton-Musiela (HJMM) equation from mathematical finance, and we refer, for example, to \cite{Bj_Sv, Bj_La, Filipovic-Teichmann, Filipovic-Teichmann-royal, Tappe-Wiener, Tappe-affine} and \cite{Tappe-Levy, Platen-Tappe, Tappe-affin-real}, where the latter three articles even deal with SPDEs driven by L\'{e}vy processes. In \cite{Projection} a similar problem has been studied; there the authors find a projection of the forward curve flow onto a finite dimensional submanifold.

The goal of this article is to go beyond stochastic invariance. More precisely, we raise the more general question when for every starting point $x \in H$, which is close to $K$, but might not be contained in $K$, the mild solution $X(\,\cdot\,;x)$ to the SPDE (\ref{SPDE-intro}) stays close to $K$.

The described property that the solutions stay close to $K$ is of interest for many applications; for example in mathematical finance, when modeling the evolution of interest rate curves. Formerly interest rate curves have been nonnegative functions, which can mathematically be treated as a stochastic invariance problem; see, for example \cite{FTT-positivity}. However, nowadays interest rate curves can also take slightly negative values, and here modeling interest rate curve evolutions being close to the cone of nonnegative functions appears to be more natural.

Furthermore, it is known that arbitrage free forward curve evolutions are typically not consistent with parametric families like the Nelson-Siegel family or the Svensson family. However, using our results we can construct arbitrage free evolutions which are close to the corresponding submanifold and derive dynamics on the submanifold which are close to the original arbitrage free dynamics.

Now, let us make the idea from above, that for every starting point $x \in H$, which is close to $K$, the mild solution solution $X(\,\cdot\,;x)$ to the SPDE (\ref{SPDE-intro}) should also be close to $K$, more precise. For this purpose, let us consider the simpler situation of a deterministic partial differential equation (PDE) of the form
\begin{align*}
\left\{
\begin{array}{rcl}
\dot{\xi}(t) & = & A \xi(t) + \alpha(\xi(t))
\\ \xi(0) & = & x
\end{array}
\right.
\end{align*}
on some Banach space $X$. Then the statement that the mild solution $\xi(\,\cdot\,;x)$ for some starting point $x \in X$ is close to $K$ means that for a fixed constant $\epsilon \geq 0$ we have
\begin{align}\label{error-function}
d_K(\xi(t;x)) \leq \Phi(d_K(x),\epsilon,t), \quad (t,x) \in \bbr_+ \times X
\end{align}
with a continuous error function $\Phi : \bbr^2 \times \bbr_+ \to \bbr_+$ such that $\Phi(0,0,t) = 0$ for all $t \in \bbr_+$, where $d_K : X \to \bbr_+$ denotes the distance function of the set $K$. We will show that the $\epsilon$-SNC ($\epsilon$-semigroup Nagumo's condition)
\begin{align}\label{liminf-PDE-pre-intro}
\liminf_{t \downarrow 0} \frac{1}{t} d_K(S_t x + t \alpha(x)) \leq \epsilon, \quad x \in K
\end{align}
is equivalent to (\ref{error-function}) with error function given by
\begin{align}\label{Phi-intro}
\Phi(d,e,t) = e^{\gamma t} d + \varphi_{\gamma}(t) e
\end{align}
for an appropriate constant $\gamma \in \bbr$, where the function $\varphi_{\gamma} : \bbr_+ \to \bbr_+$ is defined as
\begin{align*}
\varphi_{\gamma}(t) := \int_0^t e^{\gamma (t-s)} ds, \quad t \in \bbr_+.
\end{align*}
We refer to Theorems \ref{thm-Jachimiak} and \ref{thm-dist-PDE} for further details.
We call condition (\ref{liminf-PDE-pre-intro}) the $\epsilon$-SNC, as for $\epsilon = 0$ this is the well-known SNC appearing in \cite{Jachimiak-note}, where a similar result has been presented. However, as pointed out in the AMS review (MR1466843), the proof is not complete. In this paper, we amend the result from \cite{Jachimiak-note}; see in particular the aforementioned Theorem \ref{thm-Jachimiak}.

Now, let us come back to SPDEs of the type (\ref{SPDE-intro}). In this situation, we will prove a similar result for the expected distance. For $\epsilon \geq 0$ we call the condition
\begin{align}\label{SSNC-intro}
\liminf_{t \downarrow 0} \frac{1}{t} d_K \big( S_t h + t ( \alpha(h) - \rho(h) + \sigma(h) u ) \big) \leq \epsilon, \quad h \in K \text{ and } u \in U_0
\end{align}
the \emph{$\epsilon$-SSNC} (\emph{$\epsilon$-stochastic semigroup Nagumo's condition}), as for $\epsilon = 0$ this is the well-known SSNC appearing in \cite{Jachimiak} and \cite{Nakayama}. In (\ref{SSNC-intro}) the mapping $\rho : H \to H$ is given by
\begin{align*}
\rho(h) := \frac{1}{2} \sum_{j=1}^{\infty} D \sigma^j(h) \sigma^j(h),
\end{align*}
see Section \ref{sec-PDEs} for further details. Fixing a finite time horizon $T > 0$, we will show that for all $\delta > 0$ and $x_0 \in H$ there are an error function $\Phi_{\delta}^{x_0} : \bbr^2 \times [0,T] \to \bbr_+$ and an open neighborhood $U(x_0,\delta) \subset H$ of $x_0$ such that for every $\epsilon \geq 0$ the $\epsilon$-SSNC (\ref{SSNC-intro}) implies
\begin{align*}
\bbe \big[ d_K(X(t;x))^2 \big]^{1/2} \leq \delta + \Phi_{\delta}^{x_0}(d_K(x),\epsilon,t), \quad (t,x) \in [0,T] \times U(x_0,\delta).
\end{align*}
We refer to Theorem \ref{thm-SPDE} for details, and to Corollaries \ref{cor-d-zero} and \ref{cor-SPDE-inv} for further consequences concerning the $L^2$-norm $\bbe[d_K(X(t;x))^2]^{1/2}$. The main idea for the proof of Theorem \ref{thm-SPDE} is to consider Wong-Zakai approximations of the mild solutions to the SPDE (\ref{SPDE-intro}), and to apply our findings about deterministic PDEs to each sample path of the Wong-Zakai approximations. Here the crucial result is Theorem \ref{thm-piecewise}, where we treat time-inhomogeneous deterministic PDEs with coefficients which are piecewise constant in time.

Later on, we will provide conditions which ensure that the $\epsilon$-SSNC (\ref{SSNC-intro}) is fulfilled in the situation where the set $K$ is a finite dimensional submanifold with boundary. Furthermore, we will present applications to the HJMM (Heath-Jarrow-Morton-Musiela) equation from mathematical finance.

The remainder of this paper is organized as follows. In Section \ref{sec-PDEs} we present our results about deterministic PDEs, and in Section \ref{sec-PDEs-inh} we add some results about time-inhomogeneous PDEs. After these preparations, in Section \ref{sec-SPDEs} we present our results about the distance between closed sets and the solutions to SPDEs. In Section \ref{sec-manifolds} we provide some consequences in the situation where the closed set is a finite dimensional submanifold with boundary. In Section \ref{sec-subspace} we consider the situation where the submanifold is a finite dimensional subspace and derive dynamics on the subspace, which are close to the original dynamics and which can be described by a finite dimensional state process. Afterwards, in Section \ref{sec-HJMM} we present applications to the HJMM (Heath-Jarrow-Morton-Musiela) equation from mathematical finance, and in Section \ref{sec-interest} to another SPDE which is also used for interest rate modeling.

\section{Partial differential equations}\label{sec-PDEs}

In this section we provide results concerning the distance between closed sets and the solutions to deterministic PDEs. Let $X$ be a Banach space, and let $A$ be the infinitesimal generator of a $C_0$-semigroup $(S_t)_{t \geq 0}$ on $X$. We consider a deterministic PDE of the form
\begin{align}\label{PDE}
\left\{
\begin{array}{rcl}
\dot{\xi}(t) & = & A \xi(t) + \alpha(\xi(t))
\\ \xi(0) & = & x,
\end{array}
\right.
\end{align}
where $\alpha : X \to X$ is a Lipschitz continuous mapping; that is, for some constant $L \geq 0$ we have
\begin{align}\label{Lip-alpha}
\| \alpha(x) - \alpha(y) \| \leq L \| x-y \| \quad \text{for all $x,y \in X$.}
\end{align}
This ensures that for each $x \in X$ the PDE (\ref{PDE}) has a unique mild solution; that is, a continuous function $\xi(\,\cdot\,;x) : \bbr_+ \to X$ such that
\begin{align*}
\xi(t;x) = S_t x + \int_0^t S_{t-s} \alpha(\xi(s;x)) ds, \quad t \in \bbr_+.
\end{align*}
Let $K \subset X$ be a closed subset. Recall that the distance function $d_K : X \to \bbr_+$ is defined as
\begin{align*}
d_K(x) := \inf_{y \in K} \| x-y \|, \quad x \in X.
\end{align*}
For $\epsilon \geq 0$ we call the condition
\begin{align}\label{liminf-PDE-pre}
\liminf_{t \downarrow 0} \frac{1}{t} d_K(S_t x + t \alpha(x)) \leq \epsilon, \quad x \in K
\end{align}
the \emph{$\epsilon$-SNC} (\emph{$\epsilon$-semigroup Nagumo's condition}), as for $\epsilon = 0$ this is the well-known SNC appearing in \cite{Jachimiak-note}. The following auxiliary results will be useful.

\begin{lemma}\label{lemma-dist-x-y}
We have $d_K(x+y) \leq d_K(x) + \| y \|$ for all $x,y \in X$.
\end{lemma}

\begin{proof}
By a straightforward calculation we have
\begin{align*}
d_K(x+y) &= \inf_{z \in K} \| (x+y) - z \| = \inf_{z \in K} \| (x - z) + y \| \leq \inf_{z \in K} \big( \| x-z \| + \| y \| \big)
\\ &= \inf_{z \in K} \| x-z \| + \| y \| = d_K(x) + \| y \|,
\end{align*}
completing the proof
\end{proof}

\begin{lemma}\label{lemma-liminf}
Let $x \in K \cap \cald(A)$ and $v \in X$ be arbitrary. Then for each $\epsilon \geq 0$ the following statements are equivalent:
\begin{enumerate}
\item[(i)] We have
\begin{align}\label{liminf-semigroup}
\liminf_{t \downarrow 0} \frac{1}{t} d_K(S_t x + tv) \leq \epsilon.
\end{align}

\item[(ii)] We have
\begin{align}\label{liminf-A}
\liminf_{t \downarrow 0} \frac{1}{t} d_K \big( x + t (Ax + v) \big) \leq \epsilon.
\end{align}
\end{enumerate}
\end{lemma}

\begin{proof}
(i) $\Rightarrow$ (ii): Using Lemma \ref{lemma-dist-x-y}, for all $t > 0$ we have
\begin{align*}
\frac{1}{t} d_K \big( x + t (Ax + v) \big) &= \frac{1}{t} d_K \bigg( S_t x + tv + t \bigg( Ax - \frac{S_t x - x}{t} \bigg) \bigg)
\\ &\leq \frac{1}{t} d_K ( S_t x + tv ) + \bigg\| Ax - \frac{S_t x - x}{t} \bigg\|,
\end{align*}
proving (\ref{liminf-A}).

\noindent(ii) $\Rightarrow$ (i): Using Lemma \ref{lemma-dist-x-y}, for all $t > 0$ we have
\begin{align*}
\frac{1}{t} d_K(S_t x + tv) &= \frac{1}{t} d_K \bigg( x + t(Ax + v) - t \bigg( Ax - \frac{S_t x - x}{t} \bigg) \bigg)
\\ &\leq \frac{1}{t} d_K \big( x + t (Ax + v) \big) + \bigg\| Ax - \frac{S_t x - x}{t} \bigg\|,
\end{align*}
proving (\ref{liminf-semigroup}).
\end{proof}

\begin{lemma}\label{lemma-PDE-SNC}
Suppose that $K \subset \cald(A)$. Then for all $\epsilon \geq 0$ the $\epsilon$-SNC (\ref{liminf-PDE-pre}) is satisfied if and only if
\begin{align}\label{SNC-D-A}
\liminf_{t \downarrow 0} \frac{1}{t} d_K \big( x + t ( Ax + \alpha(x) \big) \leq \epsilon, \quad x \in K.
\end{align}
\end{lemma}

\begin{proof}
This is an immediate consequence of Lemma \ref{lemma-liminf}.
\end{proof}

For each $\gamma \in \bbr$ we define the function $\varphi_{\gamma} : \bbr_+ \to \bbr_+$ as
\begin{align}\label{varphi-gamma}
\varphi_{\gamma}(t) := \int_0^t e^{\gamma (t-s)} ds, \quad t \in \bbr_+.
\end{align}
Then, for $\gamma \neq 0$ we have
\begin{align*}
\varphi_{\gamma}(t) = \frac{e^{\gamma t} - 1}{\gamma}, \quad t \in \bbr_+,
\end{align*}
and in case $\gamma = 0$ we have
\begin{align*}
\varphi_0(t) = t, \quad t \in \bbr_+.
\end{align*}

\begin{remark}\label{rem-varphi}
For each $\gamma \in \bbr$ the mapping $t \mapsto \varphi_{\gamma}(t)$ is strictly increasing and smooth with $\varphi_{\gamma}(0) = 0$ and $\varphi_{\gamma}'(0) = 1$.
\end{remark}

\begin{proposition}\label{prop-1}
Let $\epsilon \geq 0$ and $\gamma \geq 0$ be such that for some $\delta > 0$ we have
\begin{align}\label{est-semiflow-pre-K}
d_K(\xi(t;x)) \leq \varphi_{\gamma}(t) \epsilon, \quad (t,x) \in [0,\delta] \times K.
\end{align}
Then we have the $\epsilon$-SNC (\ref{liminf-PDE-pre}).
\end{proposition}

\begin{proof}
We can proceed as in the proof on page 183 in \cite{Jachimiak-note}. Let $x \in K$ be arbitrary. Using Lemma \ref{lemma-dist-x-y}, for all $t \in (0,\delta]$ we have
\begin{align*}
\frac{1}{t} d_K (S_t x + t \alpha(x)) &\leq \frac{1}{t} d_K (\xi(t;x)) + \frac{1}{t} \| S_t x + t \alpha(x) - \xi(t;x) \|
\\ &\leq \epsilon \frac{\varphi_{\gamma}(t)}{t} + \bigg\| \alpha(x) - \frac{1}{t} \int_0^t S_{t-s} \alpha(\xi(s;x)) ds \bigg\|.
\end{align*}
Since $\alpha$ is continuous, we obtain
\begin{align*}
\lim_{t \downarrow 0} \bigg\| \alpha(x) - \frac{1}{t} \int_0^t S_{t-s} \alpha(\xi(s;x)) ds \bigg\| = 0.
\end{align*}
Thus, taking into account Remark \ref{rem-varphi}, we arrive at the $\epsilon$-SNC (\ref{liminf-PDE-pre}).
\end{proof}

We continue with an auxiliary result.

\begin{lemma}\label{lemma-BB}
Let $(X,\rho)$ be a complete metric space, and let $(T_t)_{t \geq 0}$ be family of mappings $T_t : X \to X$ such that the following conditions are fulfilled:
\begin{itemize}
\item $T_{t+s} = T_t T_s$ for all $t,s \geq 0$.

\item There exist $\beta \in \bbr$ and $M > 0$ such that
\begin{align*}
\rho(T_t u,T_t v) \leq M e^{\beta t} \rho(u,v) \quad \text{for all $t \geq 0$ and $u,v \in X$.}
\end{align*}
\item The mapping $\bbr_+ \to X$, $t \mapsto T_t u$ is continuous for each $u \in X$.
\end{itemize}
Let $K \subset X$ be a closed subset, and let $\epsilon \geq 0$ be a constant such that
\begin{align*}
\liminf_{t \downarrow 0} \frac{1}{t} d_K(T_t x) \leq \epsilon \quad \text{for all $x \in K$.}
\end{align*}
Then we have
\begin{align*}
d_K(T_t x) \leq M \big( e^{\beta t} d_K(x) + \epsilon \varphi_{\beta}(t) \big) \quad \text{for all $t \geq 0$ and $x \in X$.}
\end{align*}
\end{lemma}

\begin{proof}
See the Lemma on page 184 in \cite{Jachimiak-note}. For the particular case $M=1$ see also \cite[Thm. 2]{BrezisBrowder}.
\end{proof}

We recall that there exist constants $M \geq 1$ and $\beta \in \bbr$ such that the semigroup $(S_t)_{t \geq 0}$ satisfies the growth estimate
\begin{align}\label{growth-semigroup-general}
\| S_t \| \leq M e^{\beta t} \quad \text{for all $t \geq 0$.}
\end{align}
For the following results, we also recall that $L \geq 0$ denotes the Lipschitz constant from (\ref{Lip-alpha}).

\begin{proposition}\label{prop-2}
Let $\epsilon \geq 0$ be such that we have the $\epsilon$-SNC (\ref{liminf-PDE-pre}). Then we have
\begin{align}\label{est-semiflow-pre-M}
d_K(\xi(t;x)) \leq M \big( e^{(\beta+ML) t} d_K(x) + \varphi_{\beta + ML}(t) \epsilon \big), \quad (t,x) \in \bbr_+ \times X.
\end{align}
\end{proposition}

\begin{proof}
As shown on pages 184/185 in \cite{Jachimiak-note}, for all $t \geq 0$ and $x,y \in X$ we have
\begin{align*}
\| \xi(t;x) - \xi(t;y) \| \leq M e^{(\beta + ML)t} \| x-y \|.
\end{align*}
Furthermore, by Lemma \ref{lemma-dist-x-y} for all $x \in K$ we have
\begin{align*}
\frac{1}{t} d_K(\xi(t;x)) \leq \frac{1}{t} d_K(S_t x + t \alpha(x)) + \bigg\| \alpha(x) - \frac{1}{t} \int_0^t S_{t-s} \alpha(\xi(s;x)) ds \bigg\|.
\end{align*}
Therefore, by the $\epsilon$-SNC (\ref{liminf-PDE-pre}) we obtain
\begin{align*}
\liminf_{t \downarrow 0} \frac{1}{t} d_K(\xi(t;x)) \leq \epsilon \quad \text{for all $x \in K$,}
\end{align*}
and hence, by Lemma \ref{lemma-BB} the result follows.
\end{proof}

Now, we are ready to prove the following result, which amends \cite[Thm. 1]{Jachimiak-note}.

\begin{theorem}\label{thm-Jachimiak}
For each $\epsilon \geq 0$ we have the implications (i) $\Rightarrow$ (ii) $\Rightarrow$ (iii) $\Rightarrow$ (iv), where:
\begin{enumerate}
\item[(i)] We have the $\epsilon$-SNC (\ref{liminf-PDE-pre}).

\item[(ii)] We have (\ref{est-semiflow-pre-M}).

\item[(iii)] There is $\delta > 0$ such that
\begin{align*}
d_K(\xi(t;x)) \leq M \varphi_{\beta + ML}(t) \epsilon, \quad (t,x) \in [0,\delta] \times K.
\end{align*}

\item[(iv)] We have the $\epsilon$-SNC (\ref{liminf-PDE-pre}) with $\epsilon$ replaced by $M \epsilon$.
\end{enumerate}
\end{theorem}

\begin{proof}
(i) $\Rightarrow$ (ii): This implication is a consequence of Proposition \ref{prop-2}.

\noindent(ii) $\Rightarrow$ (iii): This implication is obvious.

\noindent(iii) $\Rightarrow$ (iv): This implication is a consequence of Proposition \ref{prop-1}.
\end{proof}

If $\alpha$ is bounded, then we obtain the following additional statement.

\begin{proposition}\label{prop-3}
Let $\epsilon \geq 0$ be such that we have the $\epsilon$-SNC (\ref{liminf-PDE-pre}). If there is a constant $B \in \bbr_+$ such that $\| \alpha(x) \| \leq B$ for all $x \in X$, then we have
\begin{align*}
d_K(\xi(t;x)) \leq M \big( e^{\beta t} d_K(x) + \varphi_{\beta}(t) (\epsilon + 2B) \big), \quad (t,x) \in \bbr_+ \times X.
\end{align*}
\end{proposition}

\begin{proof}
By Lemma \ref{lemma-dist-x-y}, for all $(t,x) \in \bbr_+ \times X$ we have
\begin{align*}
d_K(S_t x) \leq d_K(S_t x + t \alpha(x)) + t \| \alpha(x) \|.
\end{align*}
Therefore, by (\ref{liminf-PDE-pre}) we have
\begin{align*}
\liminf_{t \downarrow 0} \frac{1}{t} d_K(S_t x) \leq \epsilon + B, \quad x \in K.
\end{align*}
By Theorem \ref{thm-Jachimiak} with $\alpha = 0$ we obtain
\begin{align*}
d_K(S_t x) \leq M \big( e^{\beta t} d_K(x) + \varphi_{\beta}(t) (\epsilon + B) \big), \quad (t,x) \in \bbr_+ \times X.
\end{align*}
Hence, by Lemma \ref{lemma-dist-x-y} for all $(t,x) \in \bbr_+ \times X$ we have
\begin{align*}
d_K(\xi(t;x)) &\leq d_K(S_t x) + \| \xi(t;x) - S_t x \|
\\ &\leq M \big( e^{\beta t} d_K(x) + \varphi_{\beta}(t) (\epsilon + B) \big) + \bigg\| \int_0^t S_{t-s} \alpha(\xi(s;x)) ds \bigg\|.
\end{align*}
Consequently, noting that for all $(t,x) \in \bbr_+ \times X$ we have
\begin{align*}
\bigg\| \int_0^t S_{t-s} \alpha(\xi(s;x)) ds \bigg\| &\leq \int_0^t \| S_{t-s} \alpha(\xi(s;x)) \| ds \leq M \int_0^t e^{\beta(t-s)} \| \alpha(\xi(s;x)) \| ds
\\ &\leq MB \int_0^t e^{\beta(t-s)} ds = MB \varphi_{\beta}(t)
\end{align*}
completes the proof.
\end{proof}

Now, consider the particular case $M=1$, where $M \geq 1$ denotes the constant appearing in the growth estimate (\ref{growth-semigroup-general}) of the semigroup $(S_t)_{t \geq 0}$. 

\begin{theorem}\label{thm-dist-PDE}
We assume that the semigroup $(S_t)_{t \geq 0}$ is pseudo-contractive; that is, there exists a constant $\beta \in \bbr$ such that
\begin{align}\label{growth-semigroup}
\| S_t \| \leq e^{\beta t} \quad \text{for all $t \geq 0$.}
\end{align}
Then for each $\epsilon \geq 0$ the following statements are equivalent:
\begin{enumerate}
\item[(i)] We have the $\epsilon$-SNC (\ref{liminf-PDE-pre}).

\item[(ii)] We have
\begin{align*}
d_K(\xi(t;x)) \leq e^{(\beta + L) t} d_K(x) + \varphi_{\beta + L}(t) \epsilon, \quad (t,x) \in \bbr_+ \times X.
\end{align*}
\item[(iii)] There exists $\delta > 0$ such that
\begin{align*}
d_K(\xi(t;x)) \leq \varphi_{\beta + L}(t) \epsilon, \quad (t,x) \in [0,\delta] \times K.
\end{align*}
\end{enumerate}
Furthermore, the following statements are true:
\begin{enumerate}
\item If $K \subset \cald(A)$, then the equivalent statements (i), (ii) and (iii) are satisfied if and only if we have (\ref{SNC-D-A}).

\item If the equivalent statements (i), (ii) and (iii) are fulfilled, and if there is a constant $B \in \bbr_+$ such that $\| \alpha(x) \| \leq B$ for all $x \in X$, then we have
\begin{align*}
d_K(\xi(t;x)) \leq e^{\beta t} d_K(x) + \varphi_{\beta}(t) (\epsilon + 2B), \quad (t,x) \in \bbr_+ \times X.
\end{align*}
\end{enumerate}
\end{theorem}

\begin{proof}
This is a consequence of Theorem \ref{thm-Jachimiak}, Lemma \ref{lemma-PDE-SNC} and Proposition \ref{prop-3}.
\end{proof}

We obtain the following consequence for the solutions of the abstract Cauchy problem.

\begin{corollary}\label{cor-dist-PDE}
We assume that the semigroup $(S_t)_{t \geq 0}$ is pseudo-contractive; that is, there exists a constant $\beta \in \bbr$ such that we have (\ref{growth-semigroup}). Then for each $\epsilon \geq 0$ the following statements are equivalent:
\begin{enumerate}
\item[(i)] We have
\begin{align*}
\liminf_{t \downarrow 0} \frac{1}{t} d_K(S_t x) \leq \epsilon, \quad x \in K.
\end{align*}

\item[(ii)] We have
\begin{align*}
d_K(S_t x) \leq e^{\beta t} d_K(x) + \varphi_{\beta}(t) \epsilon, \quad (t,x) \in \bbr_+ \times X.
\end{align*}

\item[(iii)] There exists $\delta > 0$ such that
\begin{align*}
d_K(S_t x) \leq \varphi_{\beta}(t) \epsilon, \quad (t,x) \in [0,\delta] \times K.
\end{align*}
\end{enumerate}
If $K \subset \cald(A)$, then the equivalent statements (i), (ii) and (iii) are satisfied if and only if we have
\begin{align*}
\liminf_{t \downarrow 0} \frac{1}{t} d_K(x + tAx) \leq \epsilon, \quad x \in K.
\end{align*}
\end{corollary}

\begin{proof}
This is an immediate consequence of Theorem \ref{thm-dist-PDE}.
\end{proof}

\begin{example}
We consider the real-valued ODE
\begin{align}\label{ODE}
\left\{
\begin{array}{rcl}
\dot{\xi}(t) & = & -\beta \xi(t)
\\ \xi(0) & = & x
\end{array}
\right.
\end{align}
for some constant $\beta \in \bbr$. Then the solutions are given by
\begin{align}\label{solutions-ex-ODE}
\xi(t;x) = x e^{-\beta t}, \quad (t,x) \in \bbr_+ \times \bbr.
\end{align}
We assume that the closed subset $K \subset \bbr$ is given by $K = [a,\infty)$ for some constant $a > 0$. Let $x \in K$ be arbitrary. We distinguish the following two cases:
\begin{itemize}
\item If $x > a$, then for $t > 0$ small enough we have
\begin{align*}
\frac{1}{t} d_K(x - t \beta x) = 0.
\end{align*}
\item Now suppose that $x=a$. If $\beta \leq 0$, then we have
\begin{align*}
\frac{1}{t} d_K(x - t \beta x) = 0, \quad t > 0,
\end{align*}
and if $\beta > 0$, then we have
\begin{align*}
\frac{1}{t} d_K(x - t \beta x) = \frac{1}{t} | (a - t \beta a) - a | = \beta a, \quad t > 0.
\end{align*}
\end{itemize}
Consequently, using Corollary \ref{cor-dist-PDE} we deduce that for each $\epsilon \geq 0$ the following statements are equivalent:
\begin{enumerate}
\item[(i)] We have $\beta^+ a \leq \epsilon$, where $\beta^+ = \max \{ \beta,0 \}$ denotes the positive part of $\beta$.

\item[(ii)] We have
\begin{align}\label{est-ex-ODE}
d_K(\xi(t;x)) \leq e^{-\beta t} d_K(x) + \varphi_{-\beta}(t) \epsilon, \quad (t,x) \in \bbr_+ \times \bbr.
\end{align}

\item[(iii)] There exists $\delta > 0$ such that
\begin{align*}
d_K(\xi(t;x)) \leq \varphi_{-\beta}(t) \epsilon, \quad (t,x) \in [0,\delta] \times K.
\end{align*}
\end{enumerate}
Now, suppose that $\beta > 0$ and choose $\epsilon := \beta a$. Then by (\ref{est-ex-ODE}) we have
\begin{align*}
d_K(\xi(t;x)) \leq (a-x)^+ e^{-\beta t} + a(1 - e^{-\beta t}), \quad (t,x) \in \bbr_+ \times \bbr.
\end{align*}
Hence, for $x \leq a$ we have
\begin{align}\label{dist-ex-ODE-case-1}
d_K(\xi(t;x)) \leq a - x e^{-\beta t}, \quad t \in \bbr_+,
\end{align}
and for $x \geq a$ we have
\begin{align}\label{dist-ex-ODE-case-2}
d_K(\xi(t;x)) \leq a(1 - e^{-\beta t}), \quad t \in \bbr_+.
\end{align}
Using that the solutions of the ODE (\ref{ODE}) are given by (\ref{solutions-ex-ODE}), we can compute the distances appearing in (\ref{dist-ex-ODE-case-1}) and (\ref{dist-ex-ODE-case-2}) explicitly. For $x \leq a$ we have
\begin{align*}
d_K(\xi(t;x)) = d_K(x e^{-\beta t}) = a - x e^{-\beta t}, \quad t \in \bbr_+,
\end{align*}
showing that the upper bound in (\ref{dist-ex-ODE-case-1}) is attained. On the other hand, for $x > a$ we have
\begin{align*}
d_K(\xi(t;x)) = d_K(x e^{-\beta t}) = 0, \quad t \in \bigg[ 0, \frac{\ln(x/a)}{\beta} \bigg],
\end{align*}
showing that in (\ref{dist-ex-ODE-case-2}) we have strict inequalities.
\end{example}

\section{Time-inhomogeneous partial differential equations}\label{sec-PDEs-inh}

So far, we have considered time-homogeneous PDEs. For our purposes, we will also need some results for time-inhomogeneous PDEs, which we provide in this section. As in Section \ref{sec-PDEs}, let $X$ be a Banach space, and let $A$ be the infinitesimal generator of a $C_0$-semigroup $(S_t)_{t \geq 0}$ on $X$. We consider a time-inhomogeneous PDE of the form
\begin{align}\label{PDE-inh}
\left\{
\begin{array}{rcl}
\dot{\xi}(t) & = & A \xi(t) + \alpha(t,\xi(t))
\\ \xi(0) & = & x
\end{array}
\right.
\end{align}
with a measurable mapping $\alpha : \bbr_+ \times X \to X$. We assume there are constants $L,C \geq 0$ such that for all $t \in \bbr_+$ and $x,y \in X$ we have
\begin{align*}
\| \alpha(t,x) - \alpha(t,y) \| &\leq L \| x-y \|,
\\ \| \alpha(t,x) \| &\leq C (1 + \| x \|).
\end{align*}
Then for all $(s,x) \in \bbr_+ \times X$ there exists a unique mild solution to the PDE (\ref{PDE-inh}); that is, a continuous function $\xi(\,\cdot\,;s,x) : [s,\infty) \to X$ such that
\begin{align*}
\xi(t;s,x) = S_{t-s} x + \int_s^t S_{t-u} \alpha(u,\xi(u;s,x)) du, \quad t \in [s,\infty).
\end{align*}

\begin{lemma}[Flow property]\label{lemma-flow}
For all $0 \leq r \leq s \leq t < \infty$ and $x \in X$ we have
\begin{align}\label{flow}
\xi(t;r,x) = \xi(t;s,\xi(s;r,x)).
\end{align}
\end{lemma}

\begin{proof}
Let $r \in \bbr_+$ and $x \in X$ be arbitrary. Then we have
\begin{align*}
\xi(s;r,x) &= S_{s-r} x + \int_r^s S_{s-u} \alpha(u,\xi(u;r,x)) du, \quad s \in [r,\infty).
\end{align*}
Now, let $s \in [r,\infty)$ be arbitrary. Then we obtain
\begin{align*}
\xi(t;r,x) &= S_{t-r} x + \int_r^t S_{t-u} \alpha(u,\xi(u;r,x)) du
\\ &= S_{t-s} S_{s-r} x + S_{t-s} \int_r^s S_{s-u} \alpha(u,\xi(u;r,x)) du + \int_s^t S_{t-u} \alpha(u,\xi(u;r,x)) du
\\ &= S_{t-s} \xi(s;r,x) + \int_s^t S_{t-u} \alpha(u,\xi(u;r,x)) du, \quad t \in [s,\infty).
\end{align*}
Hence, by the uniqueness of solutions to the PDE (\ref{PDE-inh}), the flow property (\ref{flow}) follows.
\end{proof}

For the next auxiliary result, recall the definition (\ref{varphi-gamma}) of the function $\varphi_{\gamma} : \bbr_+ \to \bbr_+$.

\begin{lemma}\label{lemma-varphi-calc}
Let $\gamma \geq 0$ be arbitrary. For all $0 \leq r \leq s \leq t < \infty$ we have
\begin{align*}
e^{\gamma(t-s)} \varphi_{\gamma}(s-r) + \varphi_{\gamma}(t-s) = \varphi_{\gamma}(t-r).
\end{align*}
\end{lemma}

\begin{proof}
The straightforward calculation
\begin{align*}
&e^{\gamma(t-s)} \varphi_{\gamma}(s-r) + \varphi_{\gamma}(t-s) = e^{\gamma(t-s)} \int_0^{s-r} e^{\gamma(s-r-u)} du + \int_0^{t-s} e^{\gamma(t-s-u)} du
\\ &= \int_0^{s-r} e^{\gamma(t-r-u)} du + \int_{s-r}^{t-r} e^{\gamma(t-r-u)} du = \int_0^{t-r} e^{\gamma(t-r-u)} du = \varphi_{\gamma}(t-r)
\end{align*}
concludes the proof.
\end{proof}

For $\gamma,\delta \in \bbr_+$ we define the function $\Phi_{\gamma,\delta} : \bbr^2 \times \bbr_+ \to \bbr$ as
\begin{align}\label{Phi-gamma-delta}
\Phi_{\gamma,\delta}(d,e,t) := e^{\gamma t} d + \varphi_{\gamma}(t) (e+\delta), \quad (d,e,t) \in \bbr^2 \times \bbr_+,
\end{align}
where the function $\varphi_{\gamma} : \bbr_+ \to \bbr_+$ was introduced in (\ref{varphi-gamma}).

\begin{lemma}\label{lemma-Phi}
For all $\gamma,\delta \in \bbr_+$ the following statements are true:
\begin{enumerate}
\item For all $d,e \in \bbr^2$ we have $\Phi(d,e,0) = d$.

\item For all $d_1,d_2 \in \bbr_+$ with $d_1 \leq d_2$ and all $(e,t) \in \bbr \times \bbr_+$ we have
\begin{align*}
\Phi_{\gamma,\delta}(d_1,e,t) \leq \Phi_{\gamma,\delta}(d_2,e,t).
\end{align*}
\item For all $(d,e) \in \bbr^2$ and all $r,s,t \in \bbr_+$ with $r \leq s \leq t$ we have
\begin{align*}
\Phi_{\gamma,\delta}(\Phi_{\gamma,\delta}(d,e,s-r),e,t-s) = \Phi_{\gamma,\delta}(d,e,t-r).
\end{align*}
\end{enumerate}
\end{lemma}

\begin{proof}
The first two statements immediately follows from the definition (\ref{Phi-gamma-delta}). Furthermore, using Lemma \ref{lemma-varphi-calc} we obtain
\begin{align*}
&\Phi_{\gamma,\delta}(\Phi_{\gamma,\delta}(d,e,s-r),e,t-s) = e^{\gamma (t-s)} \Phi_{\gamma,\delta}(d,e,s-r) + \varphi_{\gamma}(t-s)(e+\delta)
\\ &= e^{\gamma (t-s)} \big( e^{\gamma (s-r)} d + \varphi_{\gamma}(s-r) (e+\delta) \big) + \varphi_{\gamma}(t-s)(e+\delta)
\\ &= e^{\gamma (t-r)} d + \varphi_{\gamma}(t-r) (e+\delta) = \Phi_{\gamma,\delta}(d,e,t-r),
\end{align*}
proving the third statement.
\end{proof}

\begin{lemma}\label{lemma-one-interval}
We assume that the semigroup $(S_t)_{t \geq 0}$ is pseudo-contractive; that is, there exists a constant $\beta \in \bbr_+$ such that
\begin{align*}
\| S_t \| \leq e^{\beta t} \quad \text{for all $t \geq 0$.}
\end{align*}
Let $0 \leq r < u < \infty$ be arbitrary. We assume there exists a function $a : X \to X$ such that
\begin{align*}
\alpha(t,x) = a(x), \quad (t,x) \in [r,u) \times X.
\end{align*}
Then for each $\epsilon \geq 0$ the following statements are equivalent:
\begin{enumerate}
\item[(i)] We have
\begin{align}\label{SNC-one-interval}
\liminf_{t \downarrow 0} \frac{1}{t} d_K(S_t x + t a(x)) \leq \epsilon, \quad x \in K.
\end{align}
\item[(ii)] For all $s,t \in [r,u]$ with $s \leq t$ we have
\begin{align*}
d_K(\xi(t;s,x)) \leq \Phi_{\beta+L,0}(d_K(x),\epsilon,t-s), \quad x \in X. 
\end{align*}
\end{enumerate}
If the previous conditions are fulfilled, and if there is a constant $B \in \bbr_+$ such that $\| a(x) \| \leq B$ for all $x \in X$, then for all $s,t \in [r,u)$ with $s \leq t$ we have
\begin{align*}
d_K(\xi(t;s,x)) \leq \Phi_{\beta,2B}(d_K(x),\epsilon,t-s), \quad x \in X.
\end{align*}
\end{lemma}

\begin{proof}
Let $s \in [r,u]$ and $x \in X$ be arbitrary. We introduce the function $\eta_s(\, \cdot \,;x) : [0,u-s] \to X$ as
\begin{align*}
\eta_s(t;x) := \xi(s+t;s,x), \quad t \in [0,u-s].
\end{align*}
Then for all $t \in [0,u-s]$ we have
\begin{align*}
\eta_s(t;x) &= \xi(s+t;s,x) 
\\ &= S_t x + \int_s^{s+t} S_{s+t-v} \alpha(v,\xi(v;s,x)) dv
\\ &= S_t x + \int_s^{s+t} S_{s+t-v} a(\xi(v;s,x)) dv
\\ &= S_t x + \int_0^{t} S_{t-v} a(\xi(s+v;s,x)) dv
\\ &= S_t x + \int_0^t S_{t-v} a(\eta_s(v;x)) dv.
\end{align*}
Therefore $\eta_s(\, \cdot \,;x)$ is the unique mild solution to the time-homogeneous PDE
\begin{align*}
\left\{
\begin{array}{rcl}
\dot{\eta}_s(t) & = & A \eta_s(t) + a(\eta_s(t))
\\ \eta_s(0) & = & x
\end{array}
\right.
\end{align*}
on the interval $[0,u-s]$. Now, we are ready to prove the claimed implications:

\noindent (i) $\Rightarrow$ (ii): Let $s,t \in [r,u]$ with $s \leq t$ and $x \in X$ be arbitrary. By Theorem \ref{thm-dist-PDE} we obtain
\begin{align*}
d_K(\xi(t;s,x)) = d_K(\eta_s(t-s;x)) \leq \Phi_{\beta+L,0}(d_K(x),\epsilon,t-s).
\end{align*}

\noindent(ii) $\Rightarrow$ (i): By assumption we have
\begin{align*}
d_K(\eta_r(t;x)) = d_K(\xi(r+t;r,x)) \leq \Phi_{\beta+L,0}(d_K(x),\epsilon,t), \quad (t,x) \in [0,u-r] \times X.
\end{align*}
Therefore, by Theorem \ref{thm-dist-PDE} we deduce (\ref{SNC-one-interval}).

\noindent The proof of the additional statement is analogous to that of the implication (i) $\Rightarrow$ (ii).
\end{proof}

\begin{theorem}\label{thm-piecewise}
We assume that the semigroup $(S_t)_{t \geq 0}$ is pseudo-contractive; that is, there exists a constant $\beta \in \bbr_+$ such that
\begin{align*}
\| S_t \| \leq e^{\beta t} \quad \text{for all $t \geq 0$.}
\end{align*}
Let $(t_n)_{n \in \bbn_0}$ be a strictly increasing sequence with $t_0 = 0$ and $t_n \to \infty$ for $n \to \infty$. We assume that for all $n \in \bbn_0$ there exists a function $a_n : X \to X$ such that
\begin{align*}
\alpha(t,x) = a_n(x), \quad (t,x) \in [t_n,t_{n+1}) \times X.
\end{align*}
Then for each $\epsilon \geq 0$ the following statements are equivalent:
\begin{enumerate}
\item[(i)] We have
\begin{align}\label{liminf-1}
\liminf_{t \downarrow 0} \frac{1}{t} d_K(S_t x + t \alpha(s,x)) \leq \epsilon, \quad (s,x) \in \bbr_+ \times K.
\end{align}
\item[(ii)] We have
\begin{align}\label{liminf-2}
\liminf_{t \downarrow 0} \frac{1}{t} d_K(S_t x + t a_n(x)) \leq \epsilon, \quad (n,x) \in \bbn_0 \times K.
\end{align}
\item[(iii)] We have
\begin{align*}
d_K(\xi(t;s,x)) \leq e^{(\beta + L) (t-s)} d_K(x) + \varphi_{\beta + L}(t-s) \epsilon, \quad 0 \leq s \leq t < \infty \text{ and } x \in X.
\end{align*}
\end{enumerate}
If the previous conditions are fulfilled, and if there is a constant $B \in \bbr_+$ such that $\| \alpha(t,x) \| \leq B$ for all $(t,x) \in \bbr_+ \times X$, then we have
\begin{align*}
d_K(\xi(t;s,x)) \leq e^{\beta (t-s)} d_K(x) + \varphi_{\beta}(t-s) (\epsilon + 2B), \quad 0 \leq s \leq t < \infty \text{ and } x \in X.
\end{align*}
\end{theorem}

\begin{proof}
(i) $\Leftrightarrow$ (ii): This equivalence is obvious.

\noindent(ii) $\Rightarrow$ (iii): Let $0 \leq s \leq t < \infty$ be arbitrary. If $s=t$, then by Lemma \ref{lemma-Phi} the claimed estimate holds true for every $x \in X$. Hence, we may assume that $s < t$. There exist $m \in \bbn$, numbers $s = s_0 < \ldots < s_m = t$, and a mapping $\pi : \{ 1,\ldots,m \} \to \bbn_0$ such that for all $j=1,\ldots,m$ we have
\begin{align*}
\alpha(t,x) = a_{\pi(j)}(x), \quad (t,x) \in [s_{j-1},s_j) \times X.
\end{align*}
By induction we will show that for all $j=0,\ldots,m$ we have
\begin{align}\label{intervals-induction}
d_K(\xi(s_j;s_0,x)) \leq \Phi_{\beta+L,0}(d_K(x),\epsilon,s_j-s_0), \quad x \in X.
\end{align}
Taking into account Lemma \ref{lemma-Phi}, the estimate (\ref{intervals-induction}) holds true for $j=0$. We proceed with the induction step $j-1 \to j$. Using the flow property (see Lemma \ref{lemma-flow}), Lemma \ref{lemma-one-interval} and Lemma \ref{lemma-Phi}, for each $x \in X$ we obtain
\begin{align*}
d_K(\xi(s_j;s_0,x)) &= d_K(\xi(s_j;s_{j-1},\xi(s_{j-1};s_0,x)))
\\ &\leq \Phi_{\beta+L,0}(d_K(\xi(s_{j-1};s_0,x)),\epsilon,s_j-s_{j-1})
\\ &\leq \Phi_{\beta+L,0}(\Phi_{\beta+L,0}(d_K(x),\epsilon,s_{j-1}-s_0),\epsilon,s_j-s_{j-1})
\\ &= \Phi_{\beta+L,0}(d_K(x),\epsilon,s_j-s_0),
\end{align*}
proving (\ref{intervals-induction}). Since $s_0 = s$ and $s_m = t$, the claimed estimate follows.

\noindent(iii) $\Rightarrow$ (ii): This implication is a consequence of Lemma \ref{lemma-one-interval}.

\noindent The proof of the additional statement is analogous to that of the implication (ii) $\Rightarrow$ (iii). 
\end{proof}

\begin{remark}
Suppose that $K \subset \cald(A)$. Then conditions (\ref{liminf-1}) and (\ref{liminf-2}) can alternatively be expressed by using Lemma \ref{lemma-liminf}.
\end{remark}

\section{Stochastic partial differential equations}\label{sec-SPDEs}

In this section we present our results about the distance between closed sets and the solutions to SPDEs. Let $(\Omega,\calf,(\calf_t)_{t \in [0,T]},\bbp)$ be a filtered probability space satisfying the usual conditions, where $T > 0$ is a finite time horizon. Let $H$ be a separable Hilbert space and let $A$ be the infinitesimal generator of a $C_0$-semigroup $(S_t)_{t \geq 0}$ on $H$. For what follows, we suppose that the semigroup $(S_t)_{t \geq 0}$ is pseudo-contractive; that is, there is a constant $\beta \in \bbr_+$ such that
\begin{align}\label{pseudo-beta}
\| S_t \| \leq e^{\beta t} \quad \text{for all $t \geq 0$.}
\end{align}
Let $U$ be a separable Hilbert space, and let $W$ be an $U$-valued $Q$-Wiener process for some nuclear, self-adjoint, positive definite linear operator $Q \in L(U)$; see \cite[Def. 4.2]{Da_Prato}. There exist an orthonormal basis $\{ e_j \}_{j \in \bbn}$ of $U$ and a sequence $(\lambda_j)_{j \in \bbn} \subset (0,\infty)$ with $\sum_{j \in \bbn} \lambda_j < \infty$ such that
\begin{align*}
Q e_j = \lambda_j e_j \quad \text{for all $j \in \bbn$.}
\end{align*}
The space $U_0 := Q^{1/2}(U)$, equipped with the inner product
\begin{align*}
\langle u,v \rangle_{U_0} := \langle Q^{-1/2}u, Q^{-1/2}v \rangle_U
\end{align*}
is another separable Hilbert space. We denote by $L_2^0(H) := L_2(U_0,H)$ the space of all Hilbert-Schmidt operators from $U_0$ into $H$. We fix the orthonormal basis $\{ g_j \}_{j \in \bbn}$ of $U_0$ given by $g_j := \sqrt{\lambda_j} e_j$ for each $j \in \bbn$, and for each $\Sigma \in L_2^0(H)$ we define $\Sigma^j \in H$ as $\Sigma^j := \Sigma g_j$ for $j \in \bbn$. According to \cite[Prop. 4.3]{Da_Prato} the sequence $(B^j)_{j \in \bbn}$ defined as
\begin{align}\label{Brownian-motion-def}
B^j := \frac{1}{\sqrt{\lambda_j}} \la W,e_j \ra_U
\end{align}
is a sequence of independent real-valued standard Wiener processes. We consider the SPDE
\begin{align}\label{SPDE}
\left\{
\begin{array}{rcl}
dX(t) & = & \big( A X(t) + \alpha(X(t)) \big) dt + \sigma(X(t)) dW(t) \medskip
\\ X(0) & = & x,
\end{array}
\right.
\end{align}
where the coefficients $\alpha : H \to H$ and $\sigma : H \to L_2^0(H)$ are assumed to be Lipschitz continuous. This ensures that for each $x \in H$ there exists a mild solution to the SPDE (\ref{SPDE}) with $X(0) = x$; that is, a continuous and adapted process $X(\,\cdot\,;x)$ such that $\bbp$-almost surely
\begin{align*}
X(t;x) = S_t x + \int_0^t S_{t-s} \alpha(X(s;x)) ds + \int_0^t S_{t-s} \sigma(X(s;x)) dW(s), \quad t \in [0,T]. 
\end{align*}
Moreover, we assume that $\sigma^j \in C_b^2(H)$ for each $j \in \bbn$. Let us define the constants
\begin{align}\label{L-j-def}
L_j := \frac{1}{2} ( \kappa_{j,2}^2 + \kappa_{j,1} \kappa_{j,3} ), \quad j \in \bbn,
\end{align}
where we have set
\begin{align}\label{kappa-1}
\kappa_{j,1} &:= \sup_{h \in H} \| \sigma^j(h) \|, \quad j \in \bbn,
\\ \label{kappa-2} \kappa_{j,2} &:= \sup_{h \in H} \| D \sigma^j(h) \|, \quad j \in \bbn,
\\ \label{kappa-3} \kappa_{j,3} &:= \sup_{h \in H} \| D^2 \sigma^j(h) \|, \quad j \in \bbn.
\end{align}
Then for each $r \in \bbn$ the mapping
\begin{align}\label{rho-r}
\rho_r : H \to H, \quad \rho_r(h) := \frac{1}{2} \sum_{j=1}^r D \sigma^j(h) \sigma^j(h)
\end{align}
is Lipschitz continuous with Lipschitz constant
\begin{align}\label{Lip-rho-r}
L_{\rho_r} := \sum_{j=1}^r L_j.
\end{align}
We assume that $\sum_{j=1}^{\infty} L_j < \infty$ and that for each $h \in H$ the limit $\lim_{r \to \infty} \rho_r(h)$ exists. Then the corresponding mapping
\begin{align}\label{rho}
\rho : H \to H, \quad \rho(h) := \lim_{r \to \infty} \rho_r(h) = \frac{1}{2} \sum_{j=1}^{\infty} D \sigma^j(h) \sigma^j(h)
\end{align}
is Lipschitz continuous with Lipschitz constant
\begin{align}\label{Lip-rho}
L_{\rho} := \sum_{j=1}^{\infty} L_j.
\end{align}
Let $K \subset H$ be a closed subset. For $\epsilon \geq 0$ we call the condition
\begin{align}\label{SSNC}
\liminf_{t \downarrow 0} \frac{1}{t} d_K \big( S_t h + t ( \alpha(h) - \rho(h) + \sigma(h) u ) \big) \leq \epsilon, \quad h \in K \text{ and } u \in U_0
\end{align}
the \emph{$\epsilon$-SSNC} (\emph{$\epsilon$-stochastic semigroup Nagumo's condition}), as for $\epsilon = 0$ this is the well-known SSNC appearing in \cite{Jachimiak} and \cite{Nakayama}.

\begin{lemma}\label{lemma-K-domain}
Suppose that $K \subset \cald(A)$. Then for all $\epsilon \geq 0$ the $\epsilon$-SSNC (\ref{SSNC}) is satisfied if and only if
\begin{align}\label{liminf-DA}
\liminf_{t \downarrow 0} \frac{1}{t} d_K \big( h + t (Ah + \alpha(h) - \rho(h) + \sigma(h)u) \big) \leq \epsilon, \quad h \in K \text{ and } u \in U_0.
\end{align}
\end{lemma}

\begin{proof}
This is an immediate consequence of Lemma \ref{lemma-liminf}.
\end{proof}

We have the following result concerning the $L^2$-norm of the distance between the closed set $K$ and the solutions to the SPDE (\ref{SPDE}).

\begin{theorem}\label{thm-SPDE}
For all $\delta > 0$ and $x_0 \in H$ there exist a continuous function $\Phi_{\delta}^{x_0} : \bbr^2 \times [0,T] \to \bbr$ satisfying
\begin{itemize}
\item $\Phi_{\delta}^{x_0}(d,e,t) \geq 0$ for all $(d,e,t) \in \bbr_+^2 \times [0,T]$,

\item $\Phi_{\delta}^{x_0}(d,e,0) = d$ for all $(d,e) \in \bbr^2$,

\item the function $t \mapsto \Phi_{\delta}^{x_0}(d,e,t)$ is increasing for all $(d,e) \in \bbr_+^2$,
\end{itemize}
and an open neighborhood $U(x_0,\delta) \subset H$ of $x_0$ such that for each $\epsilon \geq 0$ the $\epsilon$-SSNC (\ref{SSNC}) implies
\begin{align}\label{E-dist}
\bbe \big[ d_K(X(t;x))^2 \big]^{1/2} \leq \delta + \Phi_{\delta}^{x_0}(d_K(x),\epsilon,t), \quad (t,x) \in [0,T] \times U(x_0,\delta).
\end{align}
Moreover, we can choose the function $\Phi_{\delta}^{x_0}$ such that it is linear in $(d,e)$. More precisely, there are continuous, increasing functions $\phi_{\delta}^{x_0},\psi_{\delta}^{x_0} : [0,T] \to \bbr_+$ with $\phi_{\delta}^{x_0}(0) = 1$ and $\psi_{\delta}^{x_0}(0) = 0$ such that
\begin{align}\label{dist-lin-structure}
\Phi_{\delta}^{x_0}(d,e,t) = \phi_{\delta}^{x_0}(t) d + \psi_{\delta}^{x_0}(t) e, \quad (d,e,t) \in \bbr^2 \times [0,T].
\end{align}
\end{theorem}

\begin{remark}\label{rem-p-norm}
By H\"{o}lder's inequality, the estimate (\ref{E-dist}) also holds true for 
\begin{align*}
\bbe \big[ d_K(X(t;x))^p \big]^{1/p} 
\end{align*}
with an arbitrary $p \in [1,2]$. Of course, this also applies to the estimates in the upcoming corollaries.
\end{remark}

\begin{remark}
Given a reference point $x_0 \in H$ and a target error $\delta > 0$, the estimate (\ref{E-dist}) shows that the $p$-th moment of the distance $d_K(X(t;x))$ between $K$ and the solution $X(t;x)$ of the SPDE (\ref{SPDE}) starting from any initial point $x$ contained in a neighborhood $U(x_0,\delta)$ of $x_0$ can be bounded from above by $\delta$ plus an ``adjustment term''. As (\ref{dist-lin-structure}) shows, this adjustment term consists of a linear combination of the distance $d_K(x)$ and $\epsilon$ from the $\epsilon$-SSNC (\ref{SSNC}).

The coefficients of $d_K(x)$ and $\epsilon$ may depend on time $t$, but they can be taken uniformly with respect to $x$ within $U(x_0,\delta)$. Actually, since the coefficients are continuous and increasing relative to the time $t$, they can be chosen to be constant with respect to $t \in [0,T]$ and $x \in U(x_0,\delta)$. The mathematical significance of Theorem \ref{thm-SPDE} can be understood from the fact that, even if a functional defined on an infinite-dimensional space is continuous, it is generally the case that it does not attain a maximum on a bounded set.

When modeling financial or natural phenomena using SPDEs, if the coefficients $\alpha$ and $\sigma$ of the SPDE are fixed and only the initial point $x$ is varied within a neighborhood of the currently observed point $x_0$ as a reference point to reevaluate the $p$-th moment of the distance $d_K(X(t;x))$, we can take the coefficients of $d_K(x)$ and $\epsilon$ as common.
\end{remark}

\begin{remark}
As we will see from the proof of Theorem \ref{thm-SPDE}, concerning the neighborhoods we have $U(x_0,\delta) \downarrow \{ x_0 \}$ as $\delta \downarrow 0$.
\end{remark}

Before we provide the proof of Theorem \ref{thm-SPDE}, let us present some examples and consequences.

\begin{example}
Consider the deterministic PDE (\ref{PDE}), and let $\delta > 0$ and $x_0 \in H$ be arbitrary.
\begin{enumerate}
\item According to Theorem \ref{thm-dist-PDE}, the estimate (\ref{E-dist}) is satisfied with $\Phi_{\delta}^{x_0} : \bbr^2 \times [0,T] \to \bbr$ given by
\begin{align*}
\Phi_{\delta}^{x_0}(d,e,t) = e^{(\beta+L)t} d + \varphi_{\beta + L}(t) e
\end{align*}
and $U(x_0,\delta) = H$. Note that the function $\Phi_{\delta}^{x_0}$ is linear in $(d,e)$, and that here the function $\Phi_{\delta}^{x_0}$ and the neighborhood $U(x_0,\delta)$ do not depend on $(x_0,\delta)$.

\item Suppose, in addition, there is a constant $B \in \bbr_+$ such that $\| \alpha(x) \| \leq B$ for all $x \in X$. Then, according to Theorem \ref{thm-dist-PDE}, the estimate (\ref{E-dist}) is satisfied with $\Phi_{\delta}^{x_0} : \bbr^2 \times [0,T] \to \bbr$ given by
\begin{align*}
\Phi_{\delta}^{x_0}(d,e,t) = e^{\beta t} d + \varphi_{\beta}(t) (e + 2B)
\end{align*}
and $U(x_0,\delta) = H$. Note that the function $\Phi_{\delta}^{x_0}$ is affine in $(d,e)$, and that here the function $\Phi_{\delta}^{x_0}$ and the neighborhood $U(x_0,\delta)$ do not depend on $(x_0,\delta)$.
\end{enumerate}
\end{example}

\begin{corollary}\label{cor-d-zero}
Suppose the $\epsilon$-SSNC (\ref{SSNC}) is satisfied for some $\epsilon > 0$. Then for all $x_0 \in H$ and $\delta > 0$ there exist a time horizon $S \in (0,T]$ and an open neighborhood $U(x_0,\delta) \subset H$ of $x_0$ such that
\begin{align*}
\bbe \big[ d_K(X(t;x))^2 \big]^{1/2} \leq \delta, \quad (t,x) \in [0,S] \times K \cap U(x_0,\delta).
\end{align*}
\end{corollary}

\begin{proof}
By Theorem \ref{thm-SPDE} there exist a continuous, increasing function $\psi_{\delta}^{x_0} : [0,T] \to \bbr_+$ with $\psi_{\delta}^{x_0}(0) = 0$ and an open neighborhood $U(x_0,\delta) \subset H$ of $x_0$ such that
\begin{align*}
\bbe \big[ d_K(X(t;x))^2 \big]^{1/2} \leq \frac{\delta}{2} + \psi_{\delta}^{x_0}(t) \epsilon, \quad (t,x) \in [0,T] \times K \cap U(x_0,\delta).
\end{align*}
Furthermore, by the continuity of $\psi_{\delta}^{x_0}$ there exists $S \in (0,T]$ such that $\psi_{\delta}^{x_0}(S) \leq \frac{\delta}{2 \epsilon}$. This completes the proof.
\end{proof}

\begin{remark}
As can be seen from the proof of Corollary \ref{cor-d-zero}, the time interval $[0,S]$ becomes larger if $\epsilon > 0$ becomes smaller.
\end{remark}

\begin{corollary}\label{cor-SPDE-inv}
Suppose that the $\epsilon$-SSNC (\ref{SSNC}) is satisfied with $\epsilon = 0$; that is, we have the SSNC
\begin{align}\label{SSNC-zero}
\liminf_{t \downarrow 0} \frac{1}{t} d_K \big( S_t h + t ( \alpha(h) - \rho(h) + \sigma(h) u ) \big) = 0, \quad h \in K \text{ and } u \in U_0.
\end{align}
Then the following statements are true:
\begin{enumerate}
\item For all $\delta > 0$ and $x_0 \in H$ there exist $\eta > 0$ and an open neighborhood $U(x_0,\delta) \subset H$ of $x_0$ such that
\begin{align*}
\bbe \big[ d_K(X(t;x))^2 \big]^{1/2} \leq \delta, \quad (t,x) \in [0,T] \times U(x_0,\delta) \cap \{ y \in H : d_K(y) \leq \eta \}.
\end{align*}
\item In particular, the set $K$ is invariant for the SPDE (\ref{SPDE}); that is, for each $x \in K$ we have $X(\,\cdot\,;x) \in K$ up to an evanescent set.
\end{enumerate}
\end{corollary}

\begin{proof}
By Theorem \ref{thm-SPDE} there exist a continuous, increasing function $\phi_{\delta}^{x_0} : [0,T] \to \bbr_+$ with $\phi_{\delta}^{x_0}(0) = 1$ and an open neighborhood $U(x_0,\delta) \subset H$ of $x_0$ such that
\begin{align*}
\bbe \big[ d_K(X(t;x))^2 \big]^{1/2} \leq \frac{\delta}{2} + \phi_{\delta}^{x_0}(t) d_K(x), \quad (t,x) \in [0,T] \times U(x_0,\delta).
\end{align*}
Therefore, choosing $\eta > 0$ as
\begin{align*}
\eta := \frac{\delta}{2 \phi_{\delta}^{x_0}(T)},
\end{align*}
the first statement follows. The second statement about the invariance of $K$ is an immediate consequence, because for each $\delta > 0$ we obtain
\begin{align*}
\bbe \big[ d_K(X(t;x))^2 \big]^{1/2} \leq \delta, \quad (t,x) \in [0,T] \times K.
\end{align*}
This concludes the proof.
\end{proof}

\begin{remark}
It is well-known that the closed set $K$ is invariant for the SPDE (\ref{SPDE}) if and only if the SSNC (\ref{SSNC-zero}) is satisfied; see \cite[Prop. 1.1]{Nakayama}. The previous Corollary \ref{cor-SPDE-inv} provides a generalization of the statement that the SSNC (\ref{SSNC-zero}) implies invariance of the set $K$.
\end{remark}

\begin{example}[Geometric Brownian motion]
Consider the real-valued SDE
\begin{align}\label{SDE-geom}
\left\{
\begin{array}{rcl}
dX(t) & = & \mu X(t) dt + \sigma X(t) dW(t)
\\ X(0) & = & x
\end{array}
\right.
\end{align}
with constants $\mu,\sigma \in \bbr$ and a real-valued Wiener process $W$. We fix an arbitrary $p \in [1,2]$ and assume that the closed subset $K \subset \bbr$ is given by $K = (-\infty,0]$. Then the SSNC (\ref{SSNC-zero}) is fulfilled, and hence, by Corollary \ref{cor-SPDE-inv}, the subset $K$ is invariant for the SDE (\ref{SDE-geom}). Furthermore, by Theorem \ref{thm-SPDE} and Remark \ref{rem-p-norm} for all $x_0 \in \bbr$ and $\delta > 0$ there exist a continuous, increasing function $\phi_{\delta}^{x_0} : [0,T] \to \bbr_+$ with $\phi_{\delta}^{x_0}(0) = 1$ and an open neighborhood $U(x_0,\delta) \subset \bbr$ of $x_0$ such that
\begin{align}\label{est-geom}
\bbe \big[ d_K(X(t;x))^p \big]^{1/p} \leq \delta + \phi_{\delta}^{x_0}(t) d_K(x), \quad (t,x) \in [0,T] \times U(x_0,\delta).
\end{align}
Note that the solutions to the SDE (\ref{SDE-geom}) are given by
\begin{align*}
X(t;x) = x \cdot \exp \bigg( \bigg( \mu - \frac{\sigma^2}{2} \bigg) t + \sigma W(t) \bigg), \quad (t,x) \in [0,T] \times \bbr.
\end{align*}
Therefore, we have
\begin{align*}
X(t;x)^p = x^p \cdot \exp \bigg( p \bigg( \mu - \frac{\sigma^2}{2} \bigg) t + p \sigma W(t) \bigg), \quad (t,x) \in [0,T] \times \bbr.
\end{align*}
Thus, we obtain
\begin{align*}
\bbe \big[ X(t;x)^p \big]^{1/p} &= x \cdot \exp \bigg( p \bigg( \mu - \frac{\sigma^2}{2} \bigg) t + \frac{p^2 \sigma^2}{2} t \bigg)^{1/p}
\\ &= x \cdot \exp \bigg( p \bigg( \mu + \frac{(p-1)\sigma^2}{2} \bigg) t \bigg)^{1/p}
\\ &= x \cdot \exp \bigg( \bigg( \mu + \frac{(p-1)\sigma^2}{2} \bigg) t \bigg), \quad (t,x) \in [0,T] \times \bbr,
\end{align*}
and hence
\begin{align*}
\bbe \big[ d_K(X(t;x))^p \big]^{1/p} = \exp \bigg( \bigg( \mu + \frac{(p-1)\sigma^2}{2} \bigg) t \bigg) d_K(x), \quad (t,x) \in [0,T] \times \bbr.
\end{align*}
Consequently, in the estimate (\ref{est-geom}) we can choose
\begin{align*}
\phi_{\delta}^{x_0}(t) = \exp \bigg( \bigg( \mu + \frac{(p-1)\sigma^2}{2} \bigg) t \bigg), \quad t \in [0,T],
\end{align*}
and $U(x_0,\delta) = \bbr$. Note that here the function $\phi_{\delta}^{x_0}$ and the neighborhood $U(x_0,\delta)$ do not depend on $(x_0,\delta)$.
\end{example}

In order to provide the proof of Theorem \ref{thm-SPDE}, let us prepare some notation and auxiliary results. By \cite[Thm. 9.1]{Da_Prato} there is a constant $C_T > 0$ such that
\begin{align}\label{initial-cond-dist}
\bbe \big[ \| X(t;x) - X(t;y) \|^2 \big] \leq C_T \| x-y \|^2 \quad \text{for all $t \in [0,T]$ and $x,y \in H$.}
\end{align}
Let $r \in \bbn$ be arbitrary. Let $G_r \subset U_0$ be the finite dimensional subspace $G_r := \lin \{ g_1,\ldots,g_r \}$, denote by $\pi_r \in L(U_0)$ the orthogonal projection on $G_r$, which is given by
\begin{align*}
\pi_r u = \sum_{j=1}^r \la u,g_j \ra_{U_0} \, g_j, \quad u \in U_0,
\end{align*}
and let $S_r : L_2^0(H) \to L_2^0(H)$ be the linear operator given by
\begin{align*}
S_r \Sigma := \Sigma \circ \pi_r, \quad \Sigma \in L_2^0(H).
\end{align*}
Note that $S_r \in L(L_2^0(H))$ with $\| S_r \| \leq 1$, and that for each $\Sigma \in L_2^0(H)$ we have
\begin{align}\label{S-R-conv}
\| S_r \Sigma - \Sigma \|_{L_2^0(H)}^2 = \sum_{j > r} \| \Sigma^j \|^2 \to 0 \quad \text{as $r \to \infty$.}
\end{align}
Furthermore, for all $j \in \bbn$ we have
\begin{align}\label{S-r-j}
(S_r \Sigma)^j = \Sigma \pi_r g_j = \Sigma^j \bbI_{\{ j \leq r \}}.
\end{align}
Now, for each $r \in \bbn$ we consider the SPDE
\begin{align}\label{SPDE-r}
\left\{
\begin{array}{rcl}
dX_r(t) & = & \big( A X_r(t) + \alpha_r(X_r(t)) \big) dt + \sigma_r(X_r(t)) dW(t)\medskip
\\ X_r(0) & = & x.
\end{array}
\right.
\end{align}
Here the mapping $\alpha_r : H \to H$ is given by
\begin{align}\label{alpha-r}
\alpha_r(h) := \alpha(h) + \rho_r(h) - \rho(h) = \alpha(h) - \frac{1}{2} \sum_{j > r} D \sigma^j(h) \sigma^j(h), \quad h \in H,
\end{align}
where we recall that $\rho_r : H \to H$ and $\rho : H \to H$ were defined in (\ref{rho-r}) and (\ref{rho}). Furthermore, the mapping $\sigma_r : H \to L_2^0(H)$ is given by
\begin{align}\label{sigma-r}
\sigma_r(h) := S_r \sigma(h), \quad h \in H.
\end{align}
Then, noting (\ref{S-r-j}), by \cite[Prop. 2.4.5]{Liu-Roeckner} the SPDE (\ref{SPDE-r}) can be expressed as
\begin{align}\label{SPDE-r-B}
\left\{
\begin{array}{rcl}
dX_r(t) & = & \big( A X_r(t) + \alpha_r(X_r(t)) \big) dt + \sum_{j=1}^r \sigma^j(X_r(t)) dB^j(t) \medskip
\\ X_r(0) & = & x,
\end{array}
\right.
\end{align}
where $B = (B^1,\ldots,B^r)$ is the finite dimensional Brownian motion given by (\ref{Brownian-motion-def}). Furthermore, since $\rho_r(h) \to \rho(h)$ we have
\begin{align*}
\lim_{r \to \infty} \| \alpha_r(h) - \alpha(h) \| = 0 \quad \text{for each $h \in H$,}
\end{align*}
and by (\ref{S-R-conv}) we have
\begin{align*}
\lim_{r \to \infty} \| \sigma_r(h) - \sigma(h) \|_{L_2^0(H)} = 0 \quad \text{for each $h \in H$.}
\end{align*}
Now, let $r \in \bbn$ be arbitrary. Recalling that (\ref{Lip-rho-r}) is a Lipschitz constant of $\rho_r$ and that (\ref{Lip-rho}) is a Lipschitz constant of $\rho$, we have
\begin{align*}
\| \alpha_r(h) - \alpha_r(g) \| \leq ( L_{\alpha} + L_{\rho} ) \| h-g \|, \quad h,g \in H,
\end{align*}
where $L_{\alpha}$ denotes a Lipschitz constant of $\alpha$. Furthermore, recalling that $\| S_r \| \leq 1$, we have
\begin{align*}
\| \sigma_r(h) - \sigma_r(g) \| \leq L_{\sigma} \| h-g \|, \quad h,g \in H,
\end{align*}
where $L_{\sigma}$ denotes a Lipschitz constant of $\sigma$. Consequently, by \cite[Thm. 3.7]{Atma-book} we have
\begin{align}\label{conv-1}
\lim_{r \to \infty} \sup_{t \in [0,T]} \bbe \big[ \| X(t;x) - X_r(t;x) \|^2 \big] = 0, \quad x \in H.
\end{align}
Now, we fix an integer $r \in \bbn$. For each $m \in \bbn$ we define the quantities
\begin{align*}
\delta_m &:= \frac{T}{m}, \quad m \in \mathbb{N},
\\ [t]_m^- &:= k \delta_m, \quad k \delta_m \leq t < (k+1)\delta_m, \quad k = 0,\ldots,m-1,
\\ [t]_m^+ &:= (k+1) \delta_m, \quad k \delta_m \leq t < (k+1)\delta_m, \quad k = 0,\ldots,m-1,
\end{align*}
and the real-valued processes $(B_m^j(t))_{t \in [0,T]}$ for $m \in \mathbb{N}$ and $j=1\ldots,r$ as
\begin{align*}
B_m^j(t) := B^j([t]_m^-) + \frac{t - [t]_m^-}{\delta_m} (B^j([t]_m^+) - B^j([t]_m^-)), \quad t \in [0,T].
\end{align*}
Note that for all $m \in \bbn$ we have
\begin{align*}
[t]_m^- \leq t < [t]_m^+, \quad t \in [0,T],
\end{align*}
and that for all $m \in \bbn$ and all $j=1\ldots,r$ we have
\begin{align}\label{derivative}
\dot{B}_m^j(t) = \frac{B^j([t]_m^+) - B^j([t]_m^-)}{\delta_m}, \quad t \in [0,T].
\end{align}
For each $m \in \bbn$ we consider the Wong-Zakai approximation $\xi_{r,m}(\cdot) = \xi_{r,m}(\cdot,\omega) : [0,T] \to H$, which is the mild solution to the deterministic PDE
\begin{align}\label{WZ-PDE-intro}
\left\{
\begin{array}{rcl}
\dot{\xi}_{r,m}(t) & = & A \xi_{r,m}(t) + \alpha_r(\xi_{r,m}(t)) - \rho_r(\xi_{r,m}(t)) + \sum_{j=1}^r \sigma^j(\xi_{r,m}(t)) \dot{B}_m^j(t) \medskip
\\ \xi_{r,m}(0) & = & x.
\end{array}
\right.
\end{align}
By \cite[Thm. 2.1]{Nakayama-Support} for every $p > 1$ we have
\begin{align}\label{conv-2}
\lim_{m \to \infty} \bbe \bigg[ \sup_{t \in [0,T]} \| X_r(t;x) - \xi_{r,m}(t;x) \|^{2p} \bigg] = 0, \quad x \in H.
\end{align}
Noting that $\alpha_r - \rho_r = \alpha - \rho$, the PDE (\ref{WZ-PDE-intro}) can equivalently be written as
\begin{align}\label{WZ-PDE-intro-2}
\left\{
\begin{array}{rcl}
\dot{\xi}_{r,m}(t) & = & A \xi_{r,m}(t) + \alpha(\xi_{r,m}(t)) - \rho(\xi_{r,m}(t)) + \sum_{j=1}^r \sigma^j(\xi_{r,m}(t)) \dot{B}_m^j(t) \medskip
\\ \xi_{r,m}(0) & = & x.
\end{array}
\right.
\end{align}
Now, we fix an integer $m \in \bbn$ and define the nonnegative random variables
\begin{align}\label{Y-WZ}
Y_m^j := \sup_{t \in [0,T]} |\dot{B}_m^j(t)|, \quad j=1,\ldots,r.
\end{align}

\begin{lemma}\label{lemma-eta-iid}
There are i.i.d. random variables $\eta_m^{j,k} \sim {\rm N}(0,1/\delta_m)$ for $j=1,\ldots,r$ and $k=1,\ldots,m$ such that $\bbp$-almost surely
\begin{align*}
Y_m^j \leq \sum_{k=1}^m |\eta_m^{j,k}| \quad \text{for all $j=1,\ldots,r$.}
\end{align*}
\end{lemma}

\begin{proof}
The random variables
\begin{align*}
\eta_m^{j,k} := \frac{1}{\delta_m} \big( B^j(k \delta_m) - B^j((k-1) \delta_m) \big), \quad j=1,\ldots,r \text{ and } k=1,\ldots,m
\end{align*}
are independent with $\eta_m^{j,k} \sim {\rm N}(0,1/\delta_m)$ for all $j=1,\ldots,r$ and $k=1,\ldots,m$. Furthermore, taking into account (\ref{derivative}) we have $\bbp$-almost surely
\begin{align*}
Y_m^j = \sup_{t \in [0,T]} |\dot{B}_m^j(t)| \leq \frac{1}{\delta_m} \sum_{k=1}^m | ( B^j(k\delta_m) - B^j((k-1) \delta_m) ) | = \sum_{k=1}^m |\eta_m^{j,k}|,
\end{align*}
completing the proof.
\end{proof}

\begin{lemma}\label{lemma-normal-moment-exp}
For every normally distributed random variable $Z \sim {\rm N}(0,\sigma^2)$ we have
\begin{align*}
\bbe[e^{|Z|}] \leq 2 \exp \bigg( \frac{\sigma^2}{2} \bigg).
\end{align*}
\end{lemma}

\begin{proof}
By the symmetry of $Z$ we have
\begin{align*}
\bbe[e^{|Z|}] = \bbe[e^Z \bbI_{\{ Z \geq 0 \}}] + \bbe[e^{-Z} \bbI_{\{ Z < 0 \}}] \leq 2 \bbe[e^Z] = 2 \exp \bigg( \frac{\sigma^2}{2} \bigg),
\end{align*}
completing the proof.
\end{proof}

By assumption, there is a constant $L > 0$ such that for all $h,g \in H$ we have
\begin{align}
\| a(h) - a(g) \| &\leq L \| h-g \|,
\\ \| \sigma^j(h) - \sigma^j(g) \| &\leq L \| h-g \|, \quad j=1,\ldots,r,
\end{align}
where $a := \alpha - \rho$. We define the nonnegative random variable
\begin{align}\label{Z-WZ}
Z_{r,m} := L \sum_{j=1}^r Y_m^j.
\end{align}

\begin{lemma}\label{lemma-dominating}
For each $u \in \bbr$ we have $\bbe[e^{u Z_{r,m}}] < \infty$.
\end{lemma}

\begin{proof}
It suffices to consider the case $u > 0$. By Lemmas \ref{lemma-eta-iid} and \ref{lemma-normal-moment-exp} we have
\begin{align*}
\bbe \big[ e^{u Z_{r,m}} \big] &= \bbe \bigg[ \exp \bigg( L u \sum_{j=1}^r Y_m^j \bigg) \bigg] \leq \bbe \bigg[ \exp \bigg( Lu \sum_{j=1}^r \sum_{k=1}^m | \eta_m^{j,k} | \bigg) \bigg]
\\ &= \prod_{j=1}^r \prod_{k=1}^m \bbe \big[ \exp \big( Lu |\eta_m^{j,k}| \big) \big] < \infty,
\end{align*}
completing the proof.
\end{proof}

Now, we set $\gamma := \beta + L$, where $\beta \geq 0$ stems from the growth estimate (\ref{pseudo-beta}) of the semigroup $(S_t)_{t \geq 0}$. Moreover, we define
\begin{align}\label{phi-def}
\phi_{r,m}(t) &:= \bbe \big[ e^{2(\gamma + Z_{r,m})t} \big]^{1/2}, \quad t \in [0,T],
\\ \label{psi-def} \psi_{r,m}(t) &:= \bbe \big[ \varphi_{\gamma + Z_{r,m}}^2(t) \big]^{1/2}, \quad t \in [0,T].
\end{align}

\begin{lemma}\label{lemma-linear-functions}
The definitions (\ref{phi-def}) and (\ref{psi-def}) provide well-defined, continuous and increasing functions $\phi_{r,m}, \psi_{r,m} : [0,T] \to \bbr_+$ with $\phi_{r,m}(0) = 1$ and $\psi_{r,m}(0) = 0$.
\end{lemma}

\begin{proof}
Note that
\begin{align*}
e^{Z_{r,m} t} \leq e^{Z_{r,m} T}, \quad t \in [0,T].
\end{align*}
Furthermore, taking into account Remark \ref{rem-varphi} we have
\begin{align*}
\varphi_{\gamma + Z_{r,m}}(t) \leq \varphi_{\gamma + Z_{r,m}}(T) = \frac{1}{\gamma + Z_{r,m}} \big( e^{(\gamma + Z_{r,m}) T} - 1 \big) \leq \frac{e^{\gamma T}}{\gamma} e^{Z_{r,m} T}, \quad t \in [0,T].
\end{align*}
Therefore, by Lemma \ref{lemma-dominating} and Lebesgue's dominated convergence theorem the statement follows.
\end{proof}

\begin{proposition}\label{prop-est-WZ}
Let $\epsilon \geq 0$ be such that the $\epsilon$-SSNC (\ref{SSNC}) is fulfilled. Then we have
\begin{align*}
\bbe \big[ d_K(\xi_{r,m}(t;x))^2 \big]^{1/2} \leq \phi_{r,m}(t) d_K(x) + \psi_{r,m}(t) \epsilon, \quad (t,x) \in [0,T] \times H.
\end{align*}
\end{proposition}

\begin{proof}
The $\epsilon$-SSNC (\ref{SSNC}) in particular implies
\begin{align}\label{SSNC-fin-dim}
\liminf_{t \downarrow 0} \frac{1}{t} d_K \bigg( S_t h + t \bigg( a(h) + \sum_{j=1}^r \sigma^j(h) u_j \bigg) \bigg) \leq \epsilon, \quad h \in K \text{ and } u \in \bbr^r.
\end{align}
Note that the PDE (\ref{WZ-PDE-intro-2}) can be expressed as
\begin{align*}
\left\{
\begin{array}{rcl}
\dot{\xi}_{r,m}(t) & = & A \xi_{r,m}(t) + b(t,\xi_{r,m}(t))
\\ \xi_{r,m}(0) & = & x,
\end{array}
\right.
\end{align*}
where the mapping $b : \Omega \times [0,T] \times H \to H$ is given by
\begin{align*}
b(t,h) := a(h) + \sum_{j=1}^r \sigma^j(h) \dot{B}_m^j(t).
\end{align*}
By (\ref{SSNC-fin-dim}) we obtain $\bbp$-almost surely
\begin{align*}
\liminf_{t \downarrow 0} \frac{1}{t} d_K ( S_t h + t b(s,h) ) \leq \epsilon, \quad (s,h) \in [0,T] \times K.
\end{align*}
Furthermore, using (\ref{Y-WZ})--(\ref{Z-WZ}), for all $h,g \in H$ we have $\bbp$-almost surely
\begin{align*}
\| b(t,h) - b(t,g) \| &\leq \| a(h) - a(g) \| + \sum_{j=1}^r \| \sigma^j(h) - \sigma^j(g) \| \, |\dot{B}_m^j(t)|
\\ &\leq L \| h-g \| + L \| h-g \| \sum_{j=1}^r Y_m^j = (L + Z_{r,m}) \| h-g \|.
\end{align*}
Hence, recalling that $\gamma = \beta + L$, by Theorem \ref{thm-piecewise} we obtain $\bbp$-almost surely
\begin{align*}
d_K(\xi_{r,m}(t;x)) \leq e^{(\gamma + Z_{r,m})t} d_K(x) + \varphi_{\gamma + Z_{r,m}}(t) \epsilon, \quad (t,x) \in [0,T] \times H.
\end{align*}
Consequently, taking the $L^2$-norm the stated result follows.
\end{proof}

Now, we are ready to provide the proof of Theorem \ref{thm-SPDE}.

\begin{proof}[Proof of Theorem \ref{thm-SPDE}]
Let $\delta > 0$ and $x_0 \in H$ be arbitrary. By (\ref{conv-1}) there exists an index $r \in \bbn$ such that
\begin{align}\label{index-for-r}
\bbe \big[ \| X(t;x_0) - X_r(t;x_0) \|^2 \big]^{1/2} \leq \frac{\delta}{4}, \quad t \in [0,T].
\end{align}
Furthermore, by (\ref{conv-2}) there exists an index $m \in \bbn$ such that
\begin{align}\label{index-for-m}
\bbe \big[ \| X_r(t;x_0) - \xi_{r,m}(t;x_0) \|^2 \big]^{1/2} \leq \frac{\delta}{4}, \quad t \in [0,T].
\end{align}
Now, we define the function $\Phi_{\delta}^{x_0} : \bbr^2 \times [0,T] \to \bbr_+$ as
\begin{align*}
\Phi_{\delta}^{x_0}(d,e,t) := \phi_{\delta}^{x_0}(t) d + \psi_{\delta}^{x_0}(t)e, \quad (d,e,t) \in \bbr^2 \times [0,T]
\end{align*}
where $\phi_{\delta}^{x_0}, \psi_{\delta}^{x_0} : [0,T] \to \bbr_+$ are given by
\begin{align}\label{phi-in-proof}
\phi_{\delta}^{x_0}(t) &:= \bbe \big[ e^{2(\gamma + Z_{r,m})t} \big]^{1/2}, \quad t \in [0,T],
\\ \label{psi-in-proof} \psi_{\delta}^{x_0}(t) &:= \bbe \big[ \varphi_{\gamma + Z_{r,m}}^2(t) \big]^{1/2}, \quad t \in [0,T].
\end{align}
Due to Lemma \ref{lemma-linear-functions}, the functions $\Phi_{\delta}^{x_0}$ and $\phi_{\delta}^{x_0}, \psi_{\delta}^{x_0}$ satisfy all the properties stated in Theorem \ref{thm-SPDE}. Now, we define the constants
\begin{align*}
C := \phi_{\delta}^{x_0}(T) \quad \text{and} \quad \eta_0 := \frac{\delta}{4C}.
\end{align*}
By (\ref{initial-cond-dist}) there exists $\eta > 0$ with $\eta \leq \eta_0$ such that
\begin{align*}
\bbe \big[ \| X(t;x) - X(t;x_0) \|^2 \big]^{1/2} \leq \frac{\delta}{4}, \quad (t,x) \in [0,T] \times U(x_0,\delta),
\end{align*}
where $U(x_0,\delta) \subset H$ denotes the open neighborhood
\begin{align*}
U(x_0,\delta) := \{ x \in H : \| x - x_0 \| < \eta \}.
\end{align*}
Now, let $\epsilon \geq 0$ be such that the $\epsilon$-SSNC (\ref{SSNC}) is fulfilled. By Proposition \ref{prop-est-WZ}, for all $(t,x) \in [0,T] \times U(x_0,\delta)$ we have
\begin{align*}
\bbe \big[ d_K(\xi_{r,m}(t;x_0))^2 \big]^{1/2} &\leq \Phi_{\delta}^{x_0}(d_K(x_0),\epsilon,t)
\\ &\leq \Phi_{\delta}^{x_0}(d_K(x),\epsilon,t) + | \Phi_{\delta}^{x_0}(d_K(x),\epsilon,t) - \Phi_{\delta}^{x_0}(d_K(x_0),\epsilon,t) |
\\ &\leq \Phi_{\delta}^{x_0}(d_K(x),\epsilon,t) + C | d_K(x) - d_K(x_0) |
\\ &\leq \Phi_{\delta}^{x_0}(d_K(x),\epsilon,t) + C \| x - x_0 \|
\\ &\leq \Phi_{\delta}^{x_0}(d_K(x),\epsilon,t) + \frac{\delta}{4}.
\end{align*}
Thus, using Lemma \ref{lemma-dist-x-y}, for all $(t,x) \in [0,T] \times U(x_0,\delta)$ we obtain
\begin{align*}
d_K(X(t;x)) &\leq \| X(t;x) - X(t;x_0) \| + \| X(t;x_0) - X_r(t;x_0) \|
\\ &\quad + \| X_r(t;x_0) - \xi_{r,m}(t;x_0) \| + d_K(\xi_{r,m}(t;x_0)),
\end{align*}
and hence, by Minkowski's inequality we get
\begin{align*}
\bbe \big[ d_K(X(t;x))^2 \big]^{1/2} &\leq \bbe \big[ \| X(t;x) - X(t;x_0) \|^2 \big]^{1/2} + \bbe \big[ \| X(t;x_0) - X_r(t;x_0) \|^2 \big]^{1/2}
\\ &\quad + \bbe \big[ \| X_r(t;x_0) - \xi_{r,m}(t;x_0) \|^2 \big]^{1/2} + \bbe \big[ d_K(\xi_{r,m}(t;x_0))^2 \big]^{1/2}
\\ &\leq \delta + \Phi_{\delta}^{x_0}(d_K(x),\epsilon,t).
\end{align*}
This completes the proof.
\end{proof}

\begin{remark}
Combining Lemmas \ref{lemma-eta-iid}--\ref{lemma-dominating} and their proofs we obtain explicit upper bounds for the functions $\phi_{\delta}^{x_0}$ and $\psi_{\delta}^{x_0}$ defined in (\ref{phi-in-proof}) and (\ref{psi-in-proof}), which depend on the indices $r = r(x_0,\delta)$ and $m = m(x_0,\delta,r)$ given according to (\ref{index-for-r}) and (\ref{index-for-m}). Generally, we do not have explicit formulas for $r = r(x_0,\delta)$ and $m = m(x_0,\delta,r)$, but at least we have the following statements:
\begin{itemize}
\item A careful inspection of the proof of \cite[Thm. 3.7]{Atma-book} shows that in our situation there is a constant $C_1 > 0$, which can be calculated explicitly, such that for each $r \in \bbn$ and each $t \in [0,T]$ we have
\begin{align*}
\bbe \big[ \| X(t;x_0) - X_r(t;x_0) \|^2 \big] &\leq C_1 \bbe \bigg[ \int_0^T \big( \| \sigma(X(s;x_0)) - \sigma_r(X(s;x_0)) \|_{L_2^0(H)}^2
\\ &\qquad \qquad \qquad + \| \alpha(X(s;x_0)) - \alpha_r(X(s;x_0)) \|^2 \big) ds \bigg].
\end{align*}
Recalling the definitions (\ref{alpha-r}) and (\ref{sigma-r}) of $\alpha_r$ and $\sigma_r$ we obtain 
\begin{align*}
\bbe \big[ \| X(t;x_0) - X_r(t;x_0) \|^2 \big] &\leq C_1 \bbe \Bigg[ \int_0^T \Bigg( \sum_{j > r} \| \sigma^j(X(s;x_0)) \|^2
\\ &\qquad + \bigg( \frac{1}{2} \sum_{j > r} \| D \sigma^j(X(s;x_0)) \sigma^j(X(s;x_0)) \| \bigg)^2 \Bigg) ds \Bigg].
\end{align*}
Hence, recalling definitions (\ref{kappa-1})--(\ref{kappa-3}) we have
\begin{align*}
\bbe \big[ \| X(t;x_0) - X_r(t;x_0) \|^2 \big] \leq C_1 T \Bigg( \sum_{j > r} \kappa_{j,1}^2 + \bigg( \frac{1}{2} \sum_{j > r} \kappa_{j,1} \kappa_{j,2} \bigg)^2 \Bigg).
\end{align*}
Now, let us additionally assume $\sum_{j=1}^{\infty} \kappa_{j,1}^2 < \infty$. Recalling that $\sum_{j=1}^{\infty} L_j < \infty$, where $(L_j)_{j \in \bbn}$ was defined in (\ref{L-j-def}), we also have $\sum_{j=1}^{\infty} \kappa_{j,2}^2 < \infty$. Hence, by the Cauchy-Schwarz inequality we obtain $\sum_{j=1}^{\infty} \kappa_{j,1} \kappa_{j,2} < \infty$, which implies that the series appearing in (\ref{rho}) is absolutely convergent for each $h \in H$, and that the convergence $\rho_r \to \rho$ is uniform. Defining the quantities
\begin{align*}
\Sigma^r := \sum_{j > r} \kappa_{j,1}^2 + \bigg( \frac{1}{2} \sum_{j > r} \kappa_{j,1} \kappa_{j,2} \bigg)^2, \quad r \in \bbn
\end{align*}
we have $\Sigma^r \to 0$ as $r \to \infty$. If the sequence $(\Sigma^r)_{r \in \bbn}$ can be computed explicitly, then we can also determine $r = r(\delta)$ satisfying (\ref{index-for-r}) explicitly. Note that in this case $r$ only depends on $\delta$, but not on $x_0$.

\item Under the additional conditions which we will impose in Section \ref{sec-subspace}, we can apply \cite[Thm. 1.3]{Nakayama-Tappe}, which in case $x_0 \in \cald(A)$ provides us with a convergence rate for the Wong-Zakai approximations. More precisely, for each $p > 1$ we obtain the existence of a constant $C_2(x_0,r) > 0$ such that for each $m \in \bbn$ we have
\begin{align*}
\bbe \big[ \| X_r(t;x_0) - \xi_{r,m}(t;x_0) \|^{2p} \big]^{\frac1{2p}}
\leq \frac{C_2(x_0,r)}{m^{\frac{p-1}{2p}}}, \quad t \in [0,T].
\end{align*}
Noting that 
$$
\left\{\frac{p}{p-1};\ p>1\right\}=(1,\infty), 
$$
together with H\"{o}lder's inequality, for each $q > 1$ we obtain the existence of a constant $C_2(x_0,r) > 0$ such that for each $m \in \bbn$ we have
\begin{align*}
&\bbe \big[ \| X_r(t;x_0) - \xi_{r,m}(t;x_0) \|^2 \big]^{1/2}
\leq \frac{C_2(x_0,r)}{m^{\frac{1}{2q}}}, \quad t \in [0,T].
\end{align*}
Consequently, choosing $m \in \bbn$ as the minimal positive integer such that
$$
m\geq\left(\frac{4C_2(x_0,r)}\delta\right)^{2q},
$$
we obtain an explicit expression for $m = m(x_0,\delta,r)$ such that (\ref{index-for-m}) is fulfilled.
\end{itemize}
\end{remark}

\section{Finite dimensional submanifolds with boundary}\label{sec-manifolds}

In this section we investigate when the $\epsilon$-SSNC (\ref{SSNC}) is satisfied in the situation where the closed set $K$ is a finite dimensional submanifold with boundary. The general mathematical framework is that of Section \ref{sec-SPDEs}. Let $\calm \subset H$ be a finite dimensional $C^1$-submanifold with boundary of $H$. More precisely, setting $m := \dim \calm$, for each $h \in \calm$ there are a subset $V \subset \bbr_+ \times \bbr^{m-1}$, which is open with respect to the relative topology, an open neighborhood $U \subset H$ of $h$, and a mapping $\phi \in C^1(V;H)$ with the following properties:
\begin{enumerate}
\item $\phi : V \to U \cap \calm$ is a homeomorphism.

\item $D \phi(y) \in L(\bbr^m,H)$ is one-to-one for all $y \in V$.
\end{enumerate}
Such a mapping $\phi$ is called a \emph{local parametrization} of $\calm$ around $h$. The boundary $\partial \calm$ is defined as the set of all points $h \in \calm$ such that $\phi^{-1}(h) \in \partial V$ for some local parametrization $\phi : V \to U \cap \calm$ around $h$, where $\partial V$ denotes the set of all points from $V$ with vanishing first coordinate; that is 
\begin{align*}
\partial V = \{ y \in V : y_1 = 0 \}.
\end{align*}
We refer to \cite[Sec. 3]{FTT-appendix} for more details about finite dimensional submanifolds with boundary. Concerning the submanifold $\calm$, we assume that it is closed as a subset of $H$.

\begin{proposition}\label{prop-manifold}
Let $\epsilon_1,\epsilon_2 \geq 0$ be arbitrary, and let $B : \calm \to H$ be a mapping. We assume that
\begin{align}\label{diff-semigroup}
&\liminf_{t \downarrow 0} \bigg\| \frac{S_t h - h}{t} - Bh \bigg\| \leq \epsilon_1,
\\ \label{manifold-alpha-b} &d_{T_h \calm} ( Bh + \alpha(h) - \rho(h) ) \leq \epsilon_2, \quad h \in \calm \setminus \partial \calm,
\\ \label{manifold-alpha-boundary-b} &d_{(T_h \calm)_+} ( Bh + \alpha(h) - \rho(h) ) \leq \epsilon_2, \quad h \in \partial \calm,
\\ \label{manifold-sigma} &\sigma^j(h) \in
\begin{cases}
T_h \calm, & h \in \calm \setminus \partial \calm,
\\ T_h \partial \calm, & h \in \partial \calm,
\end{cases}
\quad \text{for all $j \in \bbn$.}
\end{align}
Then the $\epsilon$-SSNC (\ref{SSNC}) is satisfied with $K := \calm$ and $\epsilon := \epsilon_1 + \epsilon_2$.
\end{proposition}

\begin{proof}
Let $h \in \calm$ and $u \in U_0$ be arbitrary. We define
\begin{align*}
x := Bh + \alpha(h) - \rho(h) \quad \text{and} \quad \eta := \sigma(h)u.
\end{align*}
Then we have
\begin{align*}
\eta = \sigma(h) u = \sigma(h) \sum_{j \in \bbn} \langle u,g_j \rangle_{U_0} g_j = \sum_{j \in \bbn} \langle u,g_j \rangle_{U_0} \sigma^j(h).
\end{align*}
Therefore, and since $T_h \calm$ and $T_h \partial \calm$ are closed subspaces of $H$, by (\ref{manifold-sigma}) we obtain
\begin{align}\label{manifold-sigma-2}
\eta \in
\begin{cases}
T_h \calm, & h \in \calm \setminus \partial \calm,
\\ T_h \partial \calm, & h \in \partial \calm.
\end{cases}
\end{align}
Let $\phi : V \to U \cap \calm$ be a local parametrization of $\calm$ around $h$, and set $y := \phi^{-1}(h) \in V$. If $h \in \partial \calm$, then we have $y \in \partial V$, which just means that $y_1 = 0$. Noting that $T_h \calm$ and $(T_h \calm)_+$ are closed, convex subsets of $H$, by (\ref{manifold-alpha-b}) and (\ref{manifold-alpha-boundary-b}) there exists $\xi \in T_h \calm$ with $\| x - \xi \| \leq \epsilon_2$, and in case $h \in \partial \calm$ we even have $\xi \in (T_h \calm)_+$. There exists $v \in \bbr^m$ such that $D \phi(y)v = \xi$, and if $h \in \partial \calm$, then we even have $v_1 \geq 0$.
Furthermore, by (\ref{manifold-sigma-2}) there exists $w \in \bbr^m$ such that $D \phi(y)w = \eta$, and if $h \in \partial \calm$, then we even have $w_1 = 0$.
Consequently, there exists $\delta > 0$ such that 
\begin{align*}
y + t(v+w) \in V \quad \text{for all $t \in [0,\delta)$.} 
\end{align*}
Using Lemma \ref{lemma-dist-x-y}, for all $t \in (0,\delta)$ we obtain
\begin{align*}
&\frac{1}{t} d_{\calm} \big( S_t h + t ( \alpha(h) - \rho(h) + \sigma(h) u ) \big)
\\ &= \frac{1}{t} d_{\calm} \big( (S_{t} h - h - t Bh ) + (h + t (x + \eta)) \big)
\\ &\leq \bigg\| \frac{S_{t} h - h}{t} - Bh \bigg\| + \frac{1}{t} d_{\calm} \big( h + t (\xi + \eta) + t (x - \xi) \big)
\\ &\leq \bigg\| \frac{S_{t} h - h}{t} - Bh \bigg\| + \| x - \xi \| + \frac{1}{t} d_{\calm} \big( \phi(y) + t (\xi + \eta) \big)
\\ &\leq \bigg\| \frac{S_{t} h - h}{t} - Bh \bigg\| + \| x - \xi \| + \bigg\| \frac{\phi(y + t(v+w)) - \phi(y)}{t} - (\xi + \eta) \bigg\|.
\end{align*}
Using (\ref{diff-semigroup}), and noting that $\| x - \xi \| \leq \epsilon_2$ and $D \phi(y) (v+w) = \xi + \eta$, this gives us
\begin{align*}
\liminf_{t \downarrow 0} \frac{1}{t} d_{\calm} \big( S_t h + t ( \alpha(h) - \rho(h) + \sigma(h) u ) \big) \leq \epsilon_1 + \epsilon_2 = \epsilon,
\end{align*}
showing that the $\epsilon$-SSNC (\ref{SSNC}) is fulfilled.
\end{proof}

\begin{corollary}\label{cor-M-in-A}
Let $\epsilon \geq 0$ be arbitrary. Suppose that $\calm \subset \cald(A)$ and
\begin{align}\label{manifold-drift-1}
&d_{T_h \calm} ( Ah + \alpha(h) - \rho(h) ) \leq \epsilon, \quad h \in \calm \setminus \partial \calm,
\\ \label{manifold-drift-2} &d_{(T_h \calm)_+} ( Ah + \alpha(h) - \rho(h) ) \leq \epsilon, \quad h \in \partial \calm,
\end{align}
as well as (\ref{manifold-sigma}). Then the $\epsilon$-SSNC (\ref{SSNC}) is satisfied with $K := \calm$.
\end{corollary}

\begin{proof}
This is an immediate consequence of Proposition \ref{prop-manifold}.
\end{proof}

\begin{remark}
Suppose that $\calm \subset \cald(A)$. In case $\epsilon = 0$ conditions (\ref{manifold-drift-1}) and (\ref{manifold-drift-2}) just mean
\begin{align}\label{manifold-drift-inv}
Ah + \alpha(h) - \rho(h) \in
\begin{cases}
T_h \calm, & h \in \calm \setminus \partial \calm,
\\ (T_h \calm)_+, & h \in \partial \calm.
\end{cases}
\end{align}
It is well-known that the conditions $\calm \subset \cald(A)$ and (\ref{manifold-sigma}), (\ref{manifold-drift-inv}) are necessary and sufficient for invariance of the submanifold $\calm$ for the SPDE (\ref{SPDE}); see, for example \cite{FTT-manifolds}.
\end{remark}

\begin{lemma}\label{lemma-dist-subspace}
Let $K \subset H$ be a closed subspace. Then we have $d_K(x) = d_K(x+y)$ for all $x \in H$ and $y \in K$.
\end{lemma}

\begin{proof}
Let $x \in H$ and $y \in K$ be arbitrary. The mapping $T \in L(K)$ given by $Tz = z-y$ for all $z \in K$ is a linear isomorphism. Therefore, we obtain
\begin{align*}
d_K(x) = \inf_{z \in K} \| x-z \| = \inf_{z \in K} \| x-(z-y) \| = \inf_{z \in K} \| (x+y)-z \| = d_K(x+y),
\end{align*}
which completes the proof.
\end{proof}

\begin{corollary}\label{cor-subspace}
Let $K$ be a finite dimensional subspace such that $K \subset \cald(A)$, and let $\epsilon \geq 0$ be arbitrary. We assume that
\begin{align}\label{drift-K}
&d_K ( Ah + \alpha(h) ) \leq \epsilon, \quad h \in K,
\\ \label{sigma-in-K} &\sigma^j(h) \in K, \quad h \in K, \quad \text{for all $j \in \bbn$.}
\end{align}
Then the $\epsilon$-SSNC (\ref{SSNC}) is satisfied.
\end{corollary}

\begin{proof}
The subspace $\calm := K$ is a finite dimensional $C^1$-submanifold of $H$, which is closed as a subset of $H$. Furthermore, we have $T_h \calm = K$ for each $h \in K$, and the boundary of $K$ is empty. Therefore, condition (\ref{manifold-sigma}) is equivalent to (\ref{sigma-in-K}). Note that condition (\ref{sigma-in-K}) implies $\rho(h) \in K$ for all $h \in K$. Therefore, by Lemma \ref{lemma-dist-subspace} conditions (\ref{manifold-drift-1}) and (\ref{manifold-drift-2}) are equivalent to (\ref{drift-K}). Consequently, applying Corollary \ref{cor-M-in-A} concludes the proof. 
\end{proof}

\begin{corollary}\label{cor-subspace-2}
Let $K$ be a finite dimensional subspace such that $K \subset \cald(A)$, and let $\epsilon \geq 0$ be arbitrary. We assume that
\begin{align*}
&Ah \in K, \quad h \in K,
\\ &d_K ( \alpha(h) ) \leq \epsilon, \quad h \in K,
\\ &\sigma^j(h) \in K, \quad h \in K, \quad \text{for all $j \in \bbn$.}
\end{align*}
Then the $\epsilon$-SSNC (\ref{SSNC}) is satisfied.
\end{corollary}

\begin{proof}
This is an immediate consequence of Lemma \ref{lemma-dist-subspace} and Corollary \ref{cor-subspace}.
\end{proof}

\section{Finite dimensional subspaces}\label{sec-subspace}

In this section we consider the situation where the submanifold is a finite dimensional subspace. We will derive dynamics on the subspace, which are close to the original dynamics and which can be described by a finite dimensional state process. Also in this section the general mathematical framework is that of Section \ref{sec-SPDEs}. Note that the domain $\cald(A)$ equipped with the graph norm
\begin{align*}
\| h \|_{\cald(A)} := \sqrt{ \| h \|^2 + \| Ah \|^2 }, \quad h \in \cald(A),
\end{align*}
where $\| \cdot \|$ denotes the norm on the Hilbert space $H$, is also a separable Hilbert space. We assume that $\alpha(\cald(A)) \subset \cald(A)$ and $\sigma(\cald(A)) \subset L_2^0(\cald(A))$, and that the restricted mappings $\alpha|_{\cald(A)}$ and $\sigma|_{\cald(A)}$ satisfy the same conditions as stated for $\alpha$ and $\sigma$ in Section \ref{sec-SPDEs}.

\begin{proposition}\label{prop-sol-DA}
For each $x \in \cald(A)$ there exists a unique mild solution $X$ to the SPDE (\ref{SPDE}) with $X(0) = x$ on the state space $(\cald(A),\| \cdot \|_{\cald(A)})$, and it is a strong solution to the SPDE (\ref{SPDE}) with $X(0) = x$ on the state space $(H,\| \cdot \|)$.
\end{proposition}

\begin{proof}
The proof is analogous to \cite[Prop. 2.8]{Nakayama-Tappe}, where this result is proven with a finite dimensional driving Wiener process.
\end{proof}

Now, let $K \subset \cald(A)$ be a finite dimensional subspace. We denote by $d_K^{\cald(A)} : \cald(A) \to \bbr_+$ the distance function with respect to the graph norm; that is
\begin{align*}
d_K^{\cald(A)}(x) = \inf_{y \in K} \| x-y \|_{\cald(A)}, \quad x \in \cald(A).
\end{align*}
For each $x \in \cald(A)$ we set $P(\,\cdot\,;x) := \pi_K^{\cald(A)} X(\,\cdot\,;x)$, where $\pi_K^{\cald(A)} : \cald(A) \to K$ denotes the orthogonal projection on $K$ in the Hilbert space $(\cald(A),\| \cdot \|_{\cald(A)})$.

\begin{lemma}\label{lemma-C3}
Let $\epsilon > 0$ be arbitrary. We assume that $K \subset \cald(A^2)$ and
\begin{align}\label{DA-1}
&d_K^{\cald(A)}(Ah + \alpha(h)) \leq \epsilon, \quad h \in K,
\\ \label{DA-2} &\sigma^j(h) \in K, \quad h \in K, \quad \text{for all $j \in \bbn$.}
\end{align}
Then for all $x_0 \in \cald(A)$ and $\delta > 0$ there exist a time horizon $S \in (0,T]$ and an open neighborhood $U(x_0,\delta) \subset \cald(A)$ of $x_0$ in the Hilbert space $(\cald(A),\| \cdot \|_{\cald(A)})$ such that
\begin{align*}
\bbe \big[ \| X(t;x) - P(t;x) \|_{\cald(A)}^2 \big]^{1/2} \leq \delta, \quad (t,x) \in [0,S] \times K \cap U(x_0,\delta).
\end{align*}
\end{lemma}

\begin{proof}
Noting (\ref{DA-1}) and (\ref{DA-2}), by Corollary \ref{cor-subspace} the $\epsilon$-SSNC (\ref{SSNC}) is fulfilled on the state space $(\cald(A),\| \cdot \|_{\cald(A)})$. Thus, by Corollary \ref{cor-d-zero} there exist a time horizon $S \in (0,T]$ and an open neighborhood $U(x_0,\delta) \subset \cald(A)$ of $x_0$ in $(\cald(A),\| \cdot \|_{\cald(A)})$ such that
\begin{align*}
\bbe \big[ d_K^{\cald(A)}(X(t;x))^2 \big]^{1/2} \leq \delta, \quad (t,x) \in [0,S] \times K \cap U(x_0,\delta).
\end{align*}
Hence, for all $(t,x) \in [0,S] \times K \cap U(x_0,\delta)$ we obtain
\begin{align*}
\bbe \big[ \| X(t;x) - P(t;x) \|_{\cald(A)}^2 \big]^{1/2} &= \bbe \big[ \| X(t;x) - \pi_K^{\cald(A)} X(t;x) \|_{\cald(A)}^2 \big]^{1/2}
\\ &= \bbe \big[ d_K^{\cald(A)}(X(t;x))^2 \big]^{1/2} \leq \delta,
\end{align*}
completing the proof.
\end{proof}

Recall that $\| \cdot \|$ denotes the norm on the Hilbert space $H$.

\begin{lemma}\label{lemma-projection-extends}
Suppose that $AK \subset \cald(A^*)$. Then $\pi_K^{\cald(A)}$ extends to a continuous linear operator $\pi_K^{\cald(A)} \in L(H,K)$ in the Hilbert space $(H,\| \cdot \|)$.
\end{lemma}

\begin{proof}
Let $\{ e_1,\ldots,e_m \}$ be an orthonormal basis of $K$ in the Hilbert space $(\cald(A),\| \cdot \|_{\cald(A)})$. Since $AK \subset \cald(A^*)$, for each $x \in \cald(A)$ we have
\begin{align*}
\pi_K^{\cald(A)}(x) = \sum_{i=1}^m \la x,e_i \ra_{\cald(A)} e_i = \sum_{i=1}^m \big( \la x,e_i \ra + \la Ax,Ae_i \ra \big) e_i = \sum_{i=1}^m \la x,e_i + A^* Ae_i \ra e_i,
\end{align*}
showing that $\pi_K^{\cald(A)}$ extends to a continuous linear operator $\pi_K^{\cald(A)} \in L(H,K)$.
\end{proof}

If $AK \subset \cald(A^*)$, then we can consider the SPDE
\begin{align}\label{SPDE-K}
\left\{
\begin{array}{rcl}
dY(t) & = & L(Y(t)) dt + \sigma(Y(t)) dW(t) \medskip
\\ Y(0) & = & x,
\end{array}
\right.
\end{align}
where, using Lemma \ref{lemma-projection-extends}, the continuous mapping $L : (\cald(A),\| \cdot \|_{\cald(A)}) \to (H,\| \cdot \|)$ is given by
\begin{align*}
L(h) = \pi_{K}^{\cald(A)} (Ah + \alpha(h)), \quad h \in \cald(A).
\end{align*}
Note that (\ref{SPDE-K}) is an SPDE in continuously embedded Hilbert spaces; see \cite{Rajeev-Tappe}.

\begin{proposition}\label{prop-K-invariant}
Suppose that $AK \subset \cald(A^*)$ and that (\ref{DA-2}) is fulfilled. Then for each $x \in K$ there exists a strong solution $Y$ to the SPDE (\ref{SPDE-K}) with $Y(0) = x$ on the state space $(K,\| \cdot \|)$. In particular, the subspace $K$ is invariant for the SPDE (\ref{SPDE-K}).
\end{proposition}

\begin{proof}
Since the subspace $K$ is finite dimensional, the linear operator $A|_K : K \to H$ is continuous. Therefore, by Lemma \ref{lemma-projection-extends} the restricted mapping $L|_K : K \to K$ is Lipschitz continuous with respect to the norm $\| \cdot \|$. Furthermore, we have $\sigma(K) \subset L_2^0(K)$. Indeed, since $K$ is a closed subspace of $H$, by (\ref{DA-2}) for all $h \in K$ and $u \in U_0$ we have
\begin{align*}
\sigma(h) u = \sigma(h) \sum_{j \in \bbn} \la u,g_j \ra_{U_0} g_j = \sum_{j \in \bbn} \la u,g_j \ra_{U_0} \sigma^j(h) \in K.
\end{align*}
Moreover, the restricted mapping $\sigma|_K : K \to L_2^0(K)$ is Lipschitz continuous with respect to the norm $\| \cdot \|$. Consequently, the result follows.
\end{proof}

\begin{remark}\label{rem-state-process}
We can construct finite dimensional state processes for the SPDE (\ref{SPDE-K}) as follows. Let $\{ h_1,\ldots,h_m \}$ be a basis of $K$. We define the linear isomorphism $\phi \in L(\bbr^m,K)$ as
\begin{align*}
\phi(z) := \sum_{i=1}^m z_i h_i, \quad z \in \bbr^m.
\end{align*}
Consider the $\bbr^m$-valued SDE
\begin{align}\label{SDE-state}
\left\{
\begin{array}{rcl}
dZ(t) & = & \ell(Z(t)) dt + a(Z(t)) dW(t) \medskip
\\ Z(0) & = & z,
\end{array}
\right.
\end{align}
where $\ell : \bbr^m \to \bbr^m$ and $a : \bbr^m \to L_2^0(\bbr^m)$ are given by
\begin{align*}
\ell(z) &:= \phi^{-1} ( L(\phi(z)) ), \quad z \in \bbr^m,
\\ a^j(z) &:= \phi^{-1} ( \sigma^j(\phi(z)) ), \quad z \in \bbr^m, \quad \text{for all $j \in \bbn$.}
\end{align*}
Let $x \in K$ be arbitrary, and set $z := \phi^{-1}(x) \in \bbr^m$. Then we have $Y = \phi(Z)$, where $Y$ denotes the strong solution to the SPDE (\ref{SPDE-K}) with $Y(0) = x$, and where $Z$ denotes the strong solution to the SDE (\ref{SDE-state}) with $Z(0) = z$.
\end{remark}

Our next result shows that the solutions to the SPDEs (\ref{SPDE}) and (\ref{SPDE-K}) are indeed close to each other.

\begin{theorem}\label{thm-dist-X-Y}
Let $\epsilon > 0$ be arbitrary. We assume $K \subset \cald(A^2)$, $AK \subset \cald(A^*)$, and that conditions (\ref{DA-1}), (\ref{DA-2}) are satisfied.
Then for all $x_0 \in \cald(A)$ and $\delta > 0$ there exist a time horizon $S \in (0,T]$ and an open neighborhood $U(x_0,\delta) \subset \cald(A)$ of $x_0$ in the Hilbert space $(\cald(A),\| \cdot \|_{\cald(A)})$ such that
\begin{align*}
\bbe \big[ \| X(t;x) - Y(t;x) \|^2 \big]^{1/2} \leq \delta, \quad (t,x) \in [0,S] \times K \cap U(x_0,\delta).
\end{align*}
\end{theorem}

\begin{proof}
There is a constant $L > 0$ such that for all $h,g \in H$ we have
\begin{align}\label{Lip-alpha-dist-X-Y}
\| \alpha(h) - \alpha(g) \| &\leq L \| h-g \|,
\\ \label{Lip-sigma-dist-X-Y} \| \sigma(h) - \sigma(g) \|_{L_2^0(H)} &\leq L \| h-g \|.
\end{align}
By Lemma \ref{lemma-projection-extends} the orthogonal projection $\pi_K^{\cald(A)}$ extends to a continuous linear operator $\pi_K^{\cald(A)} \in L(H,K)$ in the Hilbert space $(H,\| \cdot \|)$. We will denote by $\| \pi_K^{\cald(A)} \|$ its operator norm with respect to $\| \cdot \|$. Moreover, since the subspace $K$ is finite dimensional, the linear operator $B := A|_K : K \to H $ is continuous. We will denote by $\| B \|$ its operator norm with respect to $\| \cdot \|$. Now, we define $\eta > 0$ as
\begin{align}\label{eta-dist-X-Y}
\eta := \frac{\delta}{2\sqrt{K_1 e^{K_2 T}}} \wedge \frac{\delta}{2},
\end{align}
where the constants $K_1, K_2 > 0$ are given by
\begin{align}\label{K1-dist-X-Y}
K_1 &:= 6\Vert\pi_K^{\cald(A)}\Vert^2T(T+TL^2+L^2),
\\ \label{K2-dist-X-Y} K_2 &:= 6\Vert\pi_K^{\cald(A)}\Vert^2(T\| B \|^2+TL^2+L^2).
\end{align}
By Lemma \ref{lemma-C3} there exist a time horizon $S \in (0,T]$ and an open neighborhood $U(x_0,\delta) \subset \cald(A)$ of $x_0$ in the Hilbert space $(\cald(A),\| \cdot \|_{\cald(A)})$ such that
\begin{align}\label{est-C3}
\bbe \big[ \| X(t;x) - P(t;x) \|_{\cald(A)}^2 \big]^{1/2} \leq \eta, \quad (t,x) \in [0,S] \times K \cap U(x_0,\delta).
\end{align}
Let $x \in K \cap U(x_0,\delta)$ be arbitrary. For convenience of notation, we set $X := X(\,\cdot\,;x)$, $P := P(\,\cdot\,;x)$ and $Y := Y(\,\cdot\,;x)$. We fix an arbitrary $t \in [0,S]$. By Proposition \ref{prop-sol-DA} the process $X$ is a strong solution to the SPDE (\ref{SPDE}) with $X(0) = x$ on the state space $(H,\| \cdot \|)$. Thus we have
\begin{align*}
P(t) = x + \int_0^t \pi_K^{\cald(A)} \big( A X(s) + \alpha(X(s)) \big) ds + \int_0^t \pi_K^{\cald(A)} \sigma(X(s)) dW(s).
\end{align*}
By Proposition \ref{prop-K-invariant} the process $Y$ is a strong solution to the SPDE (\ref{SPDE-K}) with $Y(0) = x$ on the state space $(K,\| \cdot \|)$. Using \cite[Prop. 2.4.5]{Liu-Roeckner} and (\ref{DA-2}) we obtain
\begin{align*}
\int_0^t\sigma(Y(s))dW(s) &= \sum_{j=1}^{\infty} \int_0^t \sigma^j(Y(s))dB^j(s)
\\ &= \sum_{j=1}^{\infty} \int_0^t\pi_K^{\cald(A)}\sigma^j(Y(s))dB^j(s) = \int_0^t\pi_K^{\cald(A)}\sigma(Y(s))dW(s).
\end{align*}
Therefore, we have
\begin{align*}
\int_0^t \big( \pi_K^{\cald(A)} \sigma(X(s)) - \sigma(Y(s)) \big) d W(s) = \int_0^t \pi_K^{\cald(A)} \big( \sigma(X(s)) - \sigma(Y(s)) \big) d W(s),
\end{align*}
which implies
\begin{align*}
P(t) - Y(t) &= \int_0^t \pi_K^{\cald(A)} \big( A X(s) - A Y(s) + \alpha(X(s)) - \alpha(Y(s)) \big) ds
\\ &\quad + \int_0^t \pi_K^{\cald(A)} \big( \sigma(X(s)) - \sigma(Y(s)) \big) d W(s).
\end{align*}
Now, we obtain
\begin{align*}
\bbe \big[ \| P(t) - Y(t) \|^2 \big] &\leq 3 T \| \pi_K^{\cald(A)} \|^2 \int_0^t \bbe \big[ \| A X(s) - A Y(s) \|^2 \big] ds
\\ &\quad + 3T \| \pi_K^{\cald(A)} \|^2 \int_0^t \bbe \big[ \| \alpha(X(s)) - \alpha(Y(s)) \|^2 \big] ds
\\ &\quad + 3 \| \pi_K^{\cald(A)} \|^2 \int_0^t \bbe \big[ \| \sigma(X(s)) - \sigma(Y(s)) \|_{L_2^0(H)}^2 \big] ds.
\end{align*}
Using the estimate $\| Ay \| \leq \| y \|_{\cald(A)}$, $y \in \cald(A)$ and recalling that the linear operator $B := A|_K : K \to H $ is continuous, by (\ref{est-C3}) we obtain
\begin{align*}
\int_0^t \bbe \big[ \| A X(s) - A Y(s) \|^2 \big] ds &\leq 2 \int_0^t \bbe \big[ \| A X(s) - A P(s) \|^2 \big] ds
\\ &\quad + 2 \int_0^t \bbe \big[ \| A P(s) - A Y(s) \|^2 \big] ds
\\ &\leq 2 \int_0^t  \bbe \big[\| X(s) - P(s) \|_{\cald(A)}^2 \big] ds
\\ &\quad + 2 \| B \|^2 \int_0^t \bbe \big[ \| P(s) - Y(s) \|^2 \big] ds
\\ &\leq 2 T \eta^2 + 2 \| B \|^2 \int_0^t \bbe \big[ \| P(s) - Y(s) \|^2 \big] ds.
\end{align*}
Furthermore, using the estimate $\| y \| \leq \| y \|_{\cald(A)}$, $y \in \cald(A)$ and (\ref{Lip-alpha-dist-X-Y}), (\ref{est-C3}) we obtain
\begin{align*}
\int_0^t \bbe \big[ \| \alpha(X(s)) - \alpha(Y(s)) \|^2 \big] ds &\leq 2 \int_0^t \bbe \big[ \| \alpha(X(s)) - \alpha(P(s)) \|^2 \big] ds
\\ &\quad + 2 \int_0^t \bbe \big[ \| \alpha(P(s)) - \alpha(Y(s)) \|^2 \big] ds
\\ &\leq 2 L^2 \int_0^t \bbe \big[\| X(s) - P(s) \|_{\cald(A)}^2 \big] ds
\\ &\quad + 2 L^2 \int_0^t \bbe \big[ \| P(s) - Y(s) \|^2 \big] ds
\\ &\leq 2 L^2 T \eta^2 + 2 L^2 \int_0^t \bbe \big[ \| P(s) - Y(s) \|^2 \big] ds,
\end{align*}
and similarly, using the estimate $\| y \| \leq \| y \|_{\cald(A)}$, $y \in \cald(A)$ and (\ref{Lip-sigma-dist-X-Y}), (\ref{est-C3}) we obtain
\begin{align*}
\int_0^t \bbe \big[ \| \sigma(X(s)) - \sigma(Y(s)) \|_{L_2^0(H)}^2 \big] ds \leq 2 L^2 T \eta^2 + 2 L^2 \int_0^t \bbe \big[ \| P(s) - Y(s) \|^2 \big] ds.
\end{align*}
Consequently, recalling (\ref{K1-dist-X-Y}) and (\ref{K2-dist-X-Y}) we obtain
\begin{align*}
\bbe \big[ \| P(t) - Y(t) \|^2 \big] \leq K_1 \eta^2 + K_2 \int_0^t \bbe \big[ \| P(s) - Y(s) \|^2 \big] ds, \quad t \in [0,S].
\end{align*}
Therefore, recalling (\ref{eta-dist-X-Y}), by Gronwall's inequality it follows that
\begin{align}\label{est-C4}
\bbe \big[ \| P(t) - Y(t) \|^2 \big] \leq K_1 \eta^2 e^{K_2 T} \leq \bigg( \frac{\delta}{2} \bigg)^2, \quad t \in [0,S].
\end{align}
Now, using the estimate $\| y \| \leq \| y \|_{\cald(A)}$, $y \in \cald(A)$ and noting that by (\ref{eta-dist-X-Y}) we have $\eta \leq \frac{\delta}{2}$, by (\ref{est-C3}) and (\ref{est-C4}) we obtain
\begin{align*}
\bbe \big[ \| X(t) - Y(t) \|^2 \big]^{1/2} \leq \bbe \big[ \| X(t) - P(t) \|^2 \big]^{1/2} + \bbe \big[ \| P(t) - Y(t) \|^2 \big]^{1/2} \leq \delta
\end{align*}
for all $t \in [0,S]$. This completes the proof.
\end{proof}

\begin{corollary}\label{cor-dist-X-Y}
Let $\epsilon > 0$ be arbitrary. We assume that $K \subset \cald(A^2)$, $\cald(A) \subset \cald(A^*)$,
\begin{align}\label{self-adj-2}
&Ah \in K, \quad h \in K,
\\ \label{self-adj-1} &d_K^{\cald(A)} ( \alpha(h) ) \leq \epsilon, \quad h \in K,
\end{align}
and that condition (\ref{DA-2}) is fulfilled. Then for all $x_0 \in \cald(A)$ and $\delta > 0$ there exist a time horizon $S \in (0,T]$ and an open neighborhood $U(x_0,\delta) \subset \cald(A)$ of $x_0$ in the Hilbert space $(\cald(A),\| \cdot \|_{\cald(A)})$ such that
\begin{align*}
\bbe \big[ \| X(t;x) - Y(t;x) \|^2 \big]^{1/2} \leq \delta, \quad (t,x) \in [0,S] \times K \cap U(x_0,\delta).
\end{align*}
\end{corollary}

\begin{proof}
By (\ref{self-adj-2}) we have $AK \subset K \subset \cald(A) \subset \cald(A^*)$. Furthermore, using Lemma \ref{lemma-dist-subspace} we have (\ref{DA-1}). Consequently, applying Theorem \ref{thm-dist-X-Y} completes the proof.
\end{proof}

\begin{remark}
Condition (\ref{self-adj-2}) means that the subspace $K$ is $A$-invariant. Examples of $A$-invariant subspaces can be constructed by solving the eigenvalue problems
\begin{align*}
A - \lambda = 0
\end{align*}
for real eigenvalues $\lambda \in \bbr$; see also \cite[Sec. 7]{Tappe-affin-real}.
\end{remark}

\section{The HJMM equation}\label{sec-HJMM}

In this section, we apply our results from the previous sections to the HJMM (Heath-Jarrow-Morton-Musiela) equation from mathematical finance. This SPDE models the term structure of interest rates in a market of zero coupon bonds.

We recall that a zero coupon bond with maturity $T$ is a financial asset that pays to the holder one monetary unit at time $T$. Its price at $t \leq T$ can be written as the continuous discounting of one unit of the domestic currency
\begin{align*}
P(t,T) = \exp \bigg( -\int_t^T f(t,s)ds \bigg),
\end{align*}
where $f(t,T)$ is the rate prevailing at time $t$ for instantaneous borrowing at time $T$, also called the forward rate for date $T$.

After transforming the original HJM (Heath-Jarrow-Morton) dynamics of the forward rates (see \cite{HJM}) by means of the Musiela parametrization $r_t(x) = f(t,t+x)$ (see \cite{Musiela} or \cite{Brace-Musiela}), the forward rates can be considered as a mild solution to the HJMM (Heath-Jarrow-Morton-Musiela) equation
\begin{align}\label{HJMM}
\left\{
\begin{array}{rcl}
dX(t) & = & \big( \frac{d}{dx} X(t) + \alpha_{\rm HJM}(X(t)) \big) dt + \sigma(X(t)) dW(t) \medskip
\\ X(0) & = & h,
\end{array}
\right.
\end{align}
which is a particular SPDE of the type (\ref{SPDE}). The state space of the HJMM equation (\ref{HJMM}) is a separable Hilbert space $H$ of forward curves $h : \bbr_+ \to \bbr$, and $d/dx$ denotes the differential operator, which is generated by the translation semigroup. In order to ensure absence of arbitrage in the bond market, we consider the HJMM equation (\ref{HJMM}) under a martingale measure $\bbq \approx \bbp$. Then the drift term $\alpha_{\rm HJM} : H \to H$ is given by
\begin{align}\label{HJM-drift}
\alpha_{\rm HJM}(h) = \sum_{j=1}^{\infty} \sigma^j(h) \int_0^{\bullet} \sigma^j(h)(\eta) d\eta, \quad h \in H.
\end{align}
We refer, e.g., to \cite{fillnm} for further details concerning the derivation of (\ref{HJMM}) and the drift condition (\ref{HJM-drift}). Furthermore, as state space we will choose the  Filipovi\'{c} space, which has been introduced in \cite{fillnm}. For this purpose, we fix a nondecreasing $C^1$-function $w : \bbr_+ \to [1,\infty)$ such that $w^{-1/3} \in \call^1(\bbr_+)$, and denote by $H$ the space of all absolutely continuous functions $h : \mathbb{R}_+
\rightarrow \mathbb{R}$ such that
\begin{align*}
\| h \|_H := \bigg( |h(0)|^2 + \int_0^{\infty} |h'(x)|^2
w(x) dx \bigg)^{1/2} < \infty.
\end{align*}
Then the translation semigroup $(S_t)_{t \geq 0}$ given by $S_t h = h(t+\bullet)$ for $h \in H$ is a $C_0$-semigroup on $H$, and the domain of the generator $d/dx$ is given by
\begin{align*}
\cald(d/dx) = \{ h \in H : h' \in H \},
\end{align*}
see \cite{fillnm} for details. Furthermore, the translation semigroup $(S_t)_{t \geq 0}$ is pseudo-contractive; see \cite[Lemma 3.5]{Benth-Kruehner}. In the sequel, we will choose the weight function
\begin{align}\label{weight-function}
w(x) = e^{\gamma x}, \quad x \in \bbr_+
\end{align}
for some constant $\gamma > 0$. Furthermore, we will assume that the volatility $\sigma$ is Lipschitz continuous and bounded. Then, under mild additional conditions, the drift $\alpha_{\rm HJM}$ is also Lipschitz continuous and bounded, which ensures the existence of a unique mild solution. We refer to \cite{FTT-positivity} for details. For what follows, it will be advantageous to switch to the following equivalent norm. Namely, for each $h \in H$ the limit $h(\infty) := \lim_{x \to \infty} h(x)$ exists; see estimate (5.3) in \cite{fillnm}. Hence, we can define the new inner product
\begin{align}\label{inner-prod-infty}
\la h,g \ra = h(\infty) g(\infty) + \int_0^{\infty} h'(x) g'(x) e^{\gamma x} dx, \quad h,g \in H,
\end{align}
which induces an equivalent norm $\| \cdot \|$; see \cite{Tehranchi}.

\subsection{Negative interest rates}

In this section we assume that $\sigma$ satisfies the assumptions stated in Section \ref{sec-SPDEs}; in addition, we even assume that $\sigma \in C^2(H;L_2^0(H))$. This ensures that \cite[Assumption 4.8]{FTT-positivity} is fulfilled. Let $K \subset H$ be the closed convex cone of all nonnegative forward curves; that is
\begin{align*}
K = \{ h \in H : h(x) \geq 0 \text{ for all } x \in \bbr_+ \}.
\end{align*}
Then the following statements are equivalent:
\begin{enumerate}
\item[(i)] $K$ is invariant for the HJMM equation (\ref{HJMM}).

\item[(ii)] We have the SSNC (\ref{SSNC-zero}).

\item[(iii)] For each $h \in K$ and each $x \in \bbr_+$ with $h(x) = 0$ we have
\begin{align*}
\sigma^j(h)(x) = 0, \quad j \in \bbn. 
\end{align*}
\end{enumerate}
This is a consequence of \cite[Prop. 1.1]{Nakayama} and \cite[Props. 4.4, 4.9]{FTT-positivity}. We assume that these equivalent conditions (i)--(iii) are fulfilled. 

\begin{corollary}
For all $\delta > 0$ and $h_0 \in H$ there exist $\eta > 0$ and an open neighborhood $U(h_0,\delta) \subset H$ of $h_0$ such that
\begin{align*}
\bbe \big[ d_K(X(t;h))^2 \big]^{1/2} \leq \delta, \quad (t,h) \in [0,T] \times U(h_0,\delta) \cap \{ g \in H : d_K(g) \leq \eta \}.
\end{align*}
In particular, the closed convex cone $K$ is invariant for the HJMM equation (\ref{HJMM}).
\end{corollary}

\begin{proof}
This is an immediate consequence of Corollary \ref{cor-SPDE-inv}.
\end{proof}

Consequently, the forward curves provided by the HJMM equation (\ref{HJMM}) are nonnegative if the initial forward curve $h$ is nonnegative as well. Classically, this was a desirable feature, but nowadays, we are also facing slightly negative interest rates. Therefore, it is more reasonable to start with an initial curve $h$, which is slightly negative. Then the forward curve evolution provided by the HJMM equation (\ref{HJMM}) is also slightly negative. Here are some examples for the initial forward curve.

\begin{example}
Let $h = -\eta$ for some small $\eta > 0$. Then we have $d_K(h) = \eta$. Indeed, note that
\begin{align*}
\| g-h \|^2 = | g(\infty) + \eta  |^2 + \int_0^{\infty} |g'(x)|^2 e^{\gamma x} dx \quad \text{for all $g \in K$,}
\end{align*}
which is minimal for $g = 0$. Therefore, we obtain
\begin{align*}
d_K(h) = \inf_{g \in K} \| h-g \| = \eta.
\end{align*}
\end{example}

For the next example, we prepare an auxiliary result.

\begin{lemma}\label{lemma-dist-negative-part}
We have $d_K(h) \leq \| h^- \|$ for each $h \in H$.
\end{lemma}

\begin{proof}
By Lemma \ref{lemma-dist-x-y} we have
\begin{align*}
d_K(h) = d_K(h^+ - h^-) \leq d_K(h^+) + \| h^- \| = \| h^- \|,
\end{align*}
where in the last step we have used that $h^+ \in K$.
\end{proof}

\begin{example}
In practice, one often considers an initial curve $h \in H$ which starts with negative values and becomes positive for some maturity date. Mathematically, this means that there is some $x_0 > 0$ such that 
\begin{align}\label{curve-behave}
h(x) < 0 \text{ for } x < x_0, \quad h(x_0) = 0 \quad \text{and} \quad h(x) > 0 \text{ for } x > x_0.
\end{align}
Since we use the equivalent norm induced by (\ref{inner-prod-infty}), applying Lemma \ref{lemma-dist-negative-part} we can estimate the distance as
\begin{align*}
d_K(h) \leq \| h^- \| = \sqrt{\int_0^{x_0} |h'(x)|^2 e^{\gamma x} dx}.
\end{align*}
For example, the European Central Bank publishes daily estimates of the forward rate curve from the Svensson family
\begin{align*}
h(x) = z_1 + (z_2 + z_3 x) e^{-z_5 x} + z_4 x e^{-z_6 x}, \quad x \in \bbr_+
\end{align*}
with parameters $z_1,z_2,z_3,z_4 \in \bbr$ and $z_5,z_6 > 0$. Empirically, it turns out that presently these forward rate curves have the described behaviour (\ref{curve-behave}). Note that $h \in H$ provided that $\gamma < 2z_5$ and $\gamma < 2z_6$.
\end{example}

\subsection{An extension of the Svensson family}

It has turned out that parametric families like the Nelson-Siegel family or the Svensson family are problematic when used for arbitrage free interest rate modeling. We refer to \cite{Bjoerk-Christensen, Filipovic-NS}, where it has been shown that there is no nontrivial interest rate model which is consistent with the Nelson-Siegel family, and to \cite[Sec. 3.7.2]{fillnm}, where it has been shown that there are only very few nontrivial interest rate models which are consistent with the Svensson family. In this section, we suggest a more flexible approach by allowing deviations from a parametric family. Let us consider the family $\{ F(\cdot,z) : z \in \bbr^5 \times (0,\infty)^2 \}$ consisting of all interest rate curves of the form
\begin{align*}
F(x,z) = z_1 + (z_2 + z_3 x) e^{-z_6 x} + (z_4 + z_5 x) e^{-z_7 x}, \quad x \in \bbr_+.
\end{align*}
Note that for $z_4 = 0$ this is the Svensson family, and for $z_4 = z_5 = z_7 = 0$ we obtain the Nelson-Siegel family. For what follows, we fix $z_6,z_7 > \frac{\gamma}{2}$, where we recall that $\gamma > 0$ stems from the weight function (\ref{weight-function}). We consider the five-dimensional subspace
\begin{align*}
K = \lin \{ h_1,h_2,h_3,h_4,h_5 \},
\end{align*}
where for each $x \in \bbr_+$ we have set
\begin{align*}
h_1(x) &= 1,
\\ h_2(x) &= e^{-z_6 x},
\\ h_3(x) &= x e^{-z_6 x},
\\ h_4(x) &= e^{-z_7 x},
\\ h_5(x) &= x e^{-z_7 x}.
\end{align*}
Note that $K \subset H$, because $z_6,z_7 > \frac{\gamma}{2}$. In fact, we even have $K \subset \cald(d/dx)$ and
\begin{align}\label{K-d-dx-invariant}
(d/dx) K \subset K.
\end{align}
We assume that the driving Wiener process $W$ in the HJMM equation (\ref{HJMM}) is real-valued, and that the volatility $\sigma : H \to H$ is constant and given by
\begin{align*}
\sigma = e^{-z_6 \cdot}.
\end{align*}
Then the drift term $\alpha_{\rm HJM} : H \to H$ given by (\ref{HJM-drift}) is also constant, and it is given by
\begin{align}\label{drift-HJM-z6}
\alpha_{\rm HJM} = \frac{1}{z_6} e^{-z_6\cdot} (1 - e^{-z_6\cdot}) = \frac{1}{z_6} ( e^{-z_6\cdot} - e^{-2z_6\cdot} ).
\end{align}

\begin{lemma}\label{lemma-cont-exp-z}
The mapping
\begin{align}\label{mapping-exp-H}
(\gamma/2,\infty) \to H, \quad z \mapsto e^{-z \cdot}
\end{align}
is locally Lipschitz continuous.
\end{lemma}

\begin{proof}
Let $y,z \in (\gamma/2,\infty)$ with $y \leq z$ be arbitrary. Then we have
\begin{align*}
| e^{-yx} - e^{-zx} | \leq \bigg( \sup_{w \in [y,z]} x e^{-wx} \bigg) |y-z| = x e^{-yx} |y-z|, \quad x \in \bbr_+.
\end{align*}
Therefore, we obtain
\begin{align*}
&\| e^{-y \cdot} - e^{-z \cdot} \|^2 = \int_0^{\infty} \big( y e^{-yx} - z e^{-zx} \big)^2 e^{\gamma x} dx
\\ &\leq 2 \int_0^{\infty} \big( ( y - z) e^{-yx} \big)^2 e^{\gamma x} dx + 2 \int_0^{\infty} \big( z (e^{-yx} - e^{-zx}) \big)^2 e^{\gamma x} dx
\\ &\leq 2 (y-z)^2 \int_0^{\infty} \big( e^{-yx} \big)^2 e^{\gamma x} dx + 2 z^2 (y-z)^2 \int_0^{\infty} \big( x e^{-yx} \big)^2 e^{\gamma x} dx
\\ &= \frac{2 (y-z)^2}{2y - \gamma} + \frac{4 z^2 (y-z)^2}{(2y - \gamma)^3} = \frac{2}{2y - \gamma} \bigg( 1 + \frac{2 z^2}{(2y - \gamma)^2} \bigg) (y-z)^2,
\end{align*}
proving that the mapping (\ref{mapping-exp-H}) is locally Lipschitz continuous.
\end{proof}

\begin{lemma}
We have
\begin{align*}
d_K ( \alpha_{\rm HJM} ) \leq \frac{1}{z_6} \| e^{-2 z_6 \cdot} - e^{-z_7 \cdot} \|.
\end{align*}
\end{lemma}

\begin{proof}
By (\ref{drift-HJM-z6}) and Lemma \ref{lemma-dist-subspace} we obtain
\begin{align*}
d_K ( \alpha_{\rm HJM} ) &= d_K \bigg( \frac{1}{z_6} e^{-2 z_6 \cdot} \bigg) \leq \bigg\| \frac{1}{z_6} e^{-2 z_6 \cdot} - \frac{1}{z_6} e^{-z_7 \cdot} \bigg\|
\\ &= \frac{1}{z_6} \| e^{-2 z_6 \cdot} - e^{-z_7 \cdot} \|,
\end{align*}
completing the proof.
\end{proof}

Hence, by Corollary \ref{cor-subspace-2} the $\epsilon$-SSNC (\ref{SSNC}) is satisfied with 
\begin{align*}
\epsilon = \frac{1}{z_6} \| e^{-2 z_6 \cdot} - e^{-z_7 \cdot} \|. 
\end{align*}
Of course, if $2z_6 = z_7$, then we have $\epsilon = 0$. Furthermore, by Lemma \ref{lemma-cont-exp-z} the constant $\epsilon \geq 0$ is small provided that $2 z_6 \approx z_7$.

\begin{corollary}
For all $h_0 \in H$ and $\delta > 0$ there exist a time horizon $S \in (0,T]$ and an open neighborhood $U(h_0,\delta) \subset H$ of $h_0$ such that
\begin{align*}
\bbe \big[ d_K(X(t;h))^2 \big]^{1/2} \leq \delta, \quad (t,h) \in [0,S] \times K \cap U(h_0,\delta).
\end{align*}
\end{corollary}

\begin{proof}
This is an immediate consequence of Corollary \ref{cor-d-zero}.
\end{proof}

\begin{corollary}
If $2z_6 = z_7$, then the subspace $K$ is invariant for the HJMM equation (\ref{HJMM}).
\end{corollary}

\begin{proof}
This is an immediate consequence of Corollary \ref{cor-SPDE-inv}.
\end{proof}

Unfortunately, we cannot apply Theorem \ref{thm-dist-X-Y} in the present situation, because the condition $(d/dx) K \subset \cald((d/dx)^*)$ is not fulfilled. This is a consequence of the following Proposition \ref{prop-inner-prod-unbounded}. First, we prepare an auxiliary result.

\begin{lemma}\label{lemma-inner-prod-unbounded}
Let $z > \frac{\gamma}{2}$ be arbitrary and let $h \in \cald(d/dx)$ be given by $h = e^{-z \cdot}$. Then for each $g \in \cald(d/dx)$ with $\lim_{x \to \infty} g'(x) e^{-(z-\gamma)x} = 0$ we have
\begin{align*}
\bigg\la \frac{d}{dx} g, \frac{d}{dx} h \bigg\ra = - z^2 \bigg( g'(0) + \frac{z-\gamma}{z} \la g,h \ra \bigg).
\end{align*}
\end{lemma}

\begin{proof}
Since we consider the inner product (\ref{inner-prod-infty}), using integration by parts we obtain
\begin{align*}
\la g',h' \ra &= \int_0^{\infty} g''(x) h''(x) e^{\gamma x} dx = z^2 \int_0^{\infty} g''(x) e^{-(z-\gamma) x} dx
\\ &= z^2 \bigg( - g'(0) + (z-\gamma) \int_0^{\infty} g'(x) e^{-(z-\gamma) x} dx \bigg)
\\ &= z^2 \bigg( - g'(0) - \frac{z-\gamma}{z} \int_0^{\infty} g'(x) h'(x) e^{\gamma x} dx \bigg) = - z^2 \bigg( g'(0) + \frac{z-\gamma}{z} \la g,h \ra \bigg).
\end{align*}
This completes the proof.
\end{proof}

\begin{proposition}\label{prop-inner-prod-unbounded}
Let $z > \frac{\gamma}{2}$ be arbitrary and let $h \in \cald(d/dx)$ be given by $h = e^{-z \cdot}$. Then the linear functional
\begin{align*}
\bigg( \cald \bigg( \frac{d}{dx} \bigg) ,\| \cdot \| \bigg) \to \bbr, \quad g \mapsto \bigg\la \frac{d}{dx} g, \frac{d}{dx} h \bigg\ra
\end{align*}
is unbounded.
\end{proposition}

\begin{proof}
We define the sequence $(f_n)_{n \in \bbn}$ of continuous functions $f_n : \bbr_+ \to \bbr$ as
\begin{align*}
f_n(x) := \sqrt{(n - n^2 x) \bbI_{[0,\frac{1}{n}]}(x)}, \quad x \in \bbr_+.
\end{align*}
Furthermore, we define the sequence $(g_n)_{n \in \bbn}$ of absolutely continuous functions $g_n : \bbr_+ \to \bbr$ as
\begin{align*}
g_n(x) := \int_0^x f_n(u) du - \int_0^{\infty} f_n(u) du, \quad x \in \bbr_+.
\end{align*}
Then we have $g_n \in \cald(d/dx)$ with $g_n(\infty) = 0$ and $g_n' = f_n$ for each $n \in \bbn$. Moreover, since we use the norm induced by the inner product (\ref{inner-prod-infty}), for each $n \in \bbn$ we have
\begin{align*}
\| g_n \|^2 = \int_0^1 g_n'(x)^2 e^{\gamma x} dx = \int_0^1 f_n(x)^2 e^{\gamma x} dx \leq e^{\gamma} \int_0^1 f_n(x)^2 dx = \frac{e^{\gamma}}{2}.
\end{align*}
Furthermore, for each $n \in \bbn$ we have
\begin{align*}
g_n'(0) = f_n(0) = \sqrt{n}
\end{align*}
as well as
$$
\lim_{x\to\infty}g'_n(x)e^{-(z-\gamma)x}
=\lim_{x\to\infty}f_n(x)e^{-(z-\gamma)x}=0,
$$
where we note that $f_n(x)=0$ for $x\ge\frac1n$. Therefore, by Lemma \ref{lemma-inner-prod-unbounded} we obtain 
$$
\langle g'_n,h'\rangle
=-z^2\left(g'_n(0)+\frac{z-\gamma}z\langle g_n,h\rangle\right)
=-z^2\left(\sqrt{n}+\frac{z-\gamma}z\langle g_n,h\rangle\right)
\to -\infty
$$
as $n\to\infty$. This completes the proof.
\end{proof}

However, in the upcoming section we will consider another approach to interest rate modeling and present an example, where Theorem \ref{thm-dist-X-Y} applies.

\section{Another approach to interest rate modeling}\label{sec-interest}

In the article \cite{Cont} a model for the term structure of interest rates, which is different from the HJMM equation (\ref{HJMM}), has been proposed. Here we consider a simplified version, where it is assumed that the fluctuation process satisfies a second order SPDE of the form
\begin{align}\label{SPDE-bond-2}
\left\{
\begin{array}{rcl}
dX(t) & = & \big( \frac{\kappa}{2} \frac{d^2}{d x^2} X(t) + \frac{d}{d x}X(t) + \alpha \big) dt + \sigma d W(t)
\medskip
\\ X(0) & = & h
\end{array}
\right.
\end{align}
with a positive constant $\kappa > 0$ and Dirichlet boundary conditions. The state space is $H = L^2((0,1),dx)$, and we can choose the generator
\begin{align*}
A = \frac{\kappa}{2} \frac{d^2}{d x^2} + \frac{d}{d x}
\end{align*}
on the domain $\cald(A) = H^2((0,1)) \cap H_0^1((0,1))$. Furthermore $\alpha,\sigma \in \cald(A)$ are constants, and $W$ is a real-valued Wiener process. We assume there is a finite index set $I \subset \bbn$ such that the subspace $K \subset \cald(A^2)$ is given by
\begin{align*}
K = \lin \{ x \mapsto \exp(- x / \kappa) \sin(n \pi x) : n \in I \}.
\end{align*}
We assume that $\alpha \notin K$ and $\sigma \in K$, and define the constant $\epsilon > 0$ as
\begin{align*}
\epsilon := d_K^{\cald(A)}(\alpha),
\end{align*}
where we recall the notation from Section \ref{sec-subspace}. We also recall the SPDE (\ref{SPDE-K}), which in the present situation reads
\begin{align}\label{SPDE-bond-Y}
\left\{
\begin{array}{rcl}
dY(t) & = & \pi_K^{\cald(A)} \big( \frac{\kappa}{2} \frac{d^2}{d x^2} Y(t) + \frac{d}{d x}Y(t) + \alpha \big) dt + \sigma d W(t)
\medskip
\\ Y(0) & = & h.
\end{array}
\right.
\end{align}

\begin{corollary}
Let $h_0 \in \cald(A)$ and $\delta > 0$ be arbitrary. Then there exist a time horizon $S \in (0,T]$ and an open neighborhood $U(h_0,\delta) \subset \cald(A)$ of $h_0$ in the Hilbert space $(\cald(A),\| \cdot \|_{\cald(A)})$ such that
\begin{align*}
\bbe \big[ \| X(t;h) - Y(t;h) \|^2 \big]^{1/2} \leq \delta, \quad (t,h) \in [0,S] \times K \cap U(h_0,\delta).
\end{align*}
\end{corollary}

\begin{proof}
We have $\cald(A^*) = \cald(A)$ and
\begin{align*}
A^* = \frac{\kappa}{2} \frac{d^2}{d x^2} - \frac{d}{d x}.
\end{align*}
Furthermore, the eigenvalue problem
\begin{align*}
\frac{\kappa}{2} u'' + u' - \lambda u = 0, \quad u(0) = u(1) = 0
\end{align*}
has the eigenvalues
\begin{align*}
\lambda_n = - \frac{1}{2 \kappa} \big( 1 + n^2 \pi^2 \kappa^2 \big), \quad n \in \bbn,
\end{align*}
with corresponding eigenfunctions
\begin{align*}
u_n(x) = \exp ( - x / \kappa ) \sin(n \pi x), \quad n \in \bbn.
\end{align*}
Therefore, applying Corollary \ref{cor-dist-X-Y} completes the proof.
\end{proof}

\begin{remark}
According to Remark \ref{rem-state-process} we can construct finite dimensional state processes for the SPDE (\ref{SPDE-bond-Y}) as follows. The finite index set $I \subset \bbn$ can be written as $I = \{ n_1,\ldots,n_m \}$ with positive integers $n_1 < \ldots < n_m$ for some $m \in \bbn$. We consider the $\bbr^m$-valued linear SDE
\begin{align}\label{SDE-state-HJMM}
\left\{
\begin{array}{rcl}
dZ(t) & = & (B Z(t) + b) dt + c \, dW(t) \medskip
\\ Z(0) & = & z.
\end{array}
\right.
\end{align}
Denoting by $\phi \in L(\bbr^m,K)$ the linear isomorphism given by
\begin{align*}
\phi(z) := \sum_{i=1}^m z_i u_{n_i}, \quad z \in \bbr^m,
\end{align*}
the quantities appearing in the SDE (\ref{SDE-state-HJMM}) are defined as follows:
\begin{itemize}
\item The matrix $B \in \bbr^{m \times m}$ is given by $B := \diag(\lambda_{n_1},\ldots,\lambda_{n_m})$.

\item The vector $b \in \bbr^m$ is given by $b := \phi^{-1}(\pi_K^{\cald(A)}(\alpha))$.

\item The vector $c \in \bbr^m$ is given by $c := \phi^{-1}(\sigma)$.
\end{itemize}
Let $h \in K$ be arbitrary, and set $z := \phi^{-1}(h) \in \bbr^m$. Then we have $Y = \phi(Z)$, where $Y$ denotes the strong solution to the SPDE (\ref{SPDE-bond-Y}) with $Y(0) = h$, and where $Z$ denotes the strong solution to the SDE (\ref{SDE-state-HJMM}) with $Z(0) = z$, which is given by
\begin{align*}
Z(t) = e^{Bt} \bigg[ z + e^{-Bt} c W(t) + \int_0^t e^{-Bs} ( b + B c W(s) ) ds \bigg], \quad t \in [0,T].
\end{align*}
\end{remark}

\begin{remark}
There are several examples of generators $A$ of $C_0$-semigroups which satisfy $\cald(A) \subset \cald(A^*)$; for example the Laplace operator $A = \Delta$. For such operators, similar examples of the subspace $K$ can be treated by solving the eigenvalue problems
\begin{align*}
A - \lambda = 0
\end{align*}
for real eigenvalues $\lambda \in \bbr$. Then Corollary \ref{cor-dist-X-Y} also applies, for example to the stochastic heat equation. We refer to \cite[Sec. 7]{Tappe-affin-real} for further examples of this kind.
\end{remark}

\end{document}